\documentclass[a4paper, reqno]{amsart}
\usepackage{hyperref,url}
\usepackage{a4wide}
\usepackage{amsthm, amssymb, amsmath, latexsym, upgreek,accents}
\usepackage{amsfonts, mathrsfs, color}
\usepackage{graphicx,color}

\usepackage{tikz}

\usepackage{booktabs}
\usepackage{textcomp}
\usepackage{subfigure}
\usepackage{soul}

 \theoremstyle{plain}
 \begingroup
 \newtheorem{theo}{Theorem}[section]
 \newtheorem{lemma}[theo]{Lemma}
 \newtheorem{propo}[theo]{Proposition}
 \newtheorem{prop}[theo]{Property}
 \newtheorem{coro}[theo]{Corollary}
 \endgroup

 \theoremstyle{definition}
 \begingroup
 \newtheorem{defi}[theo]{Definition}

 \endgroup

 \theoremstyle{remark}
 \begingroup
 \newtheorem{rem}[theo]{Remark}
 \endgroup

 \numberwithin{equation}{section}

\mathchardef\emptyset="001F

\newcommand{\dx}{\,dx}
\newcommand{\dy}{\,dy}
\newcommand{\ie}{{; \it i.e., }}

\newcommand{\HH}{{\mathcal H}^{n-1}}

\newcommand{\I}{\mathcal{I}}
\newcommand{\B}{\mathscr{B}}
\newcommand{\T}{\mathcal{T}}

\let\e= \varepsilon
\newcommand{\R}{{\mathbb R}}
\newcommand{\Z}{{\mathbb Z}}

\newcommand{\N}{{\mathbb N}}
\newcommand{\Sph}{{\mathbb S}}

\newcommand{\Om}{\Omega}
\newcommand{\om}{\omega}

\newcommand{\mres}{\mathbin{\vrule height 1.6ex depth 0pt width 0.13ex\vrule height 0.13ex depth 0pt width 1.3ex}}
\newcommand{\wto}{\rightharpoonup}

\DeclareMathOperator\supp{supp}

\author[X. Pellet]{Xavier Pellet}
\address[X. Pellet]{Department of Mathematical Sciences, University of Bath, Bath, United Kingdom}
\email[]{X.P.J.Pellet@bath.ac.uk}
\author[L. Scardia]{Lucia Scardia}
\address[L. Scardia]{Department of Mathematics, Heriot-Watt University, United Kingdom}
\email[]{L.Scardia@hw.ac.uk}
\author[C.I. Zeppieri]{Caterina Ida Zeppieri}
\address[C.I. Zeppieri]{Angewandte Mathematik, WWU M\"unster, Germany}
\email[]{caterina.zeppieri@uni-muenster.de}

\title[Free-discontinuity functionals in random perforated domains]{Stochastic homogenisation of free-discontinuity functionals in random perforated domains}

\begin{document}

\maketitle

\begin{abstract} In this paper we study the asymptotic behaviour of a family of random free-discontinuity energies $E_\e$ defined on a randomly perforated domain, as $\e$ goes to zero. The functionals $E_\e$ model the energy associated to displacements of porous random materials that can develop cracks. To gain compactness for sequences of displacements with bounded energies, we need to overcome the lack of equi-coerciveness of the functionals. We do so by means of an extension result, under the assumption that the random perforations cannot come too close to one another. The limit energy is then obtained in two steps. As a first step we apply a general result of stochastic convergence of free-discontinuity functionals to a modified, coercive version of $E_\e$. Then the effective volume and surface energy densities are identified by means of a careful limit procedure.
\end{abstract}

\maketitle

\medskip

{\small
\noindent 

\medskip

\noindent 
}

\section{Introduction}

\noindent In this paper we prove a stochastic homogenisation result for free-discontinuity functionals defined on randomly perforated domains. More  precisely we consider the functionals $E_\e$ given by 
	\begin{equation}\label{intro:Ee}
		E_\e(\omega)(u,A)=
					\int_{A\setminus \e K(\omega)}f\left(\omega,\frac{x}{\e}, \nabla u\right)dx + \int_{S_u\cap (A\setminus \e K(\omega))}g\left(\omega,\frac{x}{\e},\nu_u\right)\, d\HH,
	\end{equation}
for $u \in SBV(A)$; here $A\subset \R^n$ is a bounded, Lipschitz domain, and $SBV(A)$ denotes the set of special functions of bounded variation in $A$. In \eqref{intro:Ee} the parameter $\om$ belongs to the sample space $\Om$ of a given probability space $(\Omega,\mathcal{T}, {P})$, whereas $\e>0$ sets the geometric scale of the problem. The integrands $f$ and $g$ are stationary random variables, thus they are to be interpreted as an ensemble of coefficients, and $K(\omega)$ denotes a collection of randomly distributed $n$-dimensional balls with random radii (see \eqref{def:P}). Since the integration in \eqref{intro:Ee} is performed only on the set $A\setminus \e K(\omega)$, the set $K(\omega)$ models a collection of randomly distributed perforations inside the material occupying the reference configuration $A$.  Energies of this type can be used to describe the elastic energy of a porous brittle random material.

In the deterministic \textit{periodic} setting, the limit behaviour of energies of type \eqref{intro:Ee} has been studied both in the case of Dirichlet conditions on the perforations \cite{FG} and in the case of natural boundary conditions \cite{BaFo, CS, FSP1}. 
Only very recently, in \cite{CDMSZ2}, the stochastic homogenisation of free-discontinuity functionals was considered, under quite general assumptions on the volume and surface integrands, and in the vector-valued case (see \cite{CDMSZ1}, and \cite{BDfV, GP} for the deterministic counterpart). In \cite{CDMSZ2}, however, the volume and surface integrands must satisfy non-degenerate lower bounds, which is not the case for $E_\e$, due to the presence of the perforations.

The study of the asymptotic behaviour of elliptic problems in \textit{randomly} perforated domains has a long history starting with the seminal work of Jikov \cite{Ji}. We refer the reader to the book \cite{JiKoOl-book} and the references therein for the classical results on this subject. More recently the random counterpart of the work by Cioranescu and Murat \cite{CiMu} has been also considered \cite{CDnonlinear, CDPoisson, GHV}. In this case, sequences $u_\e$ of equi-bounded energy can be trivially extended to zero inside $\e K(\om)$, due to the homogeneous \emph{Dirichlet boundary conditions}, and hence can be assumed from the onset to satisfy a priori bounds on the whole domain. In the Dirichlet setting the main difficulty in the analysis lies then in the characterisation of the limiting ``capacitary'' term. Since in this case no extension result for the $u_\e$ is needed, the assumptions on the geometry of the perforations can be rather mild \cite{GHV}.

In this paper we assume instead that sequences $u_\e$ of equi-bounded energy satisfy \emph{natural boundary conditions} on the perforations, which makes the compactness of minimising sequences subtle. In this setting the classical way to obtain compactness is to \textit{extend} the functions $u_\e$ inside the perforations in a way that keeps the functionals on the extended functions comparable with the functionals on $u_\e$. In the periodic case, and for Sobolev functions, the use of extension theorems as a powerful technique to treat degenerate problems is due to Khruslov \cite{Khru}, Cioranescu and Paulin \cite{CSJP}, and to Tartar \cite{Pec}. In that setting, the most general extension result is due to Acerbi, Chiad\`o Piat, Dal Maso, and Percivale \cite{ACDP}, and has been proved under minimal assumptions on the geometry of the periodic perforations, which in particular can be connected. 

In the random case a common approach to the homogenisation of perforated (or porous) materials is to assume the existence of an extension operator as a property of the domain (see, \textit{e.g.}, \cite[Chapter 8]{JiKoOl-book}). More precisely, it is often assumed that the perforated domain $A\setminus \e K(\om)$ is a random set, that it is open and connected, that its density (namely the expectation of its characteristic function) is strictly positive, and that there exists an extension operator from the perforated to the full domain. These assumptions guarantee compactness of sequences with equi-bounded energies, and allow to prove existence of the $\Gamma$-limit, and non-degeneracy of the limit energy. Alternatively, simplified random geometries are considered, for which one can prove directly that the random domain satisfies the assumptions above. This is the case for a class of disperse media, the so-called \emph{random spherical structure}\ie a system of many hard sphere particles. In the simplest case of such structure the domain has an underlying $\e$-periodic grid, and in each $\e$-cell the random perforation is a ball - with random radius and centre - which is $\e \delta$-separated from the boundary of the cell where it is contained, for a given $\delta>0$. A more general geometry is given by the case where the spherical holes are $2\e\delta$-separated from one another, but no underlying periodic ``safety'' grid is postulated. For random spherical structures it is shown, \textit{e.g.}, in \cite[Section 8.4]{JiKoOl-book} that if the spherical holes are not too close to one another, then the density of the domain is strictly positive, and some extension operator exists in the Sobolev setting. 

Our approach is in the same spirit, and we now explain it in some detail.
\subsection{Overview of the main results} In what follows we give a brief overview of the main results contained in this paper: An extension result for special functions of bounded variation in a randomly perforated domain, and the $\Gamma$-convergence of the functionals $E_\e$ in \eqref{intro:Ee}.

\smallskip

\textbf{The extension property in $SBV$.} The geometry we consider for the randomly perforated domain is the following: We assume that the perforations $K(\om)$ are disjoint balls of random centres and radii, and we require that the minimal distance between any two of them is $2\delta$, where $\delta>0$ is independent of the realisation $\om$. In other words, not only the perforations are separated, but also their $\delta$-neighbourhoods are so. 
Our first main result is an extension property for this class of domains in $SBV$ (Lemma \ref{t:ext-cell} and Theorem \ref{stochastic-extension-lemma-SBV}). We recall that the existence of an extension operator in $SBV$, for the Mumford-Shah functional, has been proved by Cagnetti and Scardia \cite{CS} in the periodic case. This result, however, cannot be applied directly to our case since the domain $A\setminus \e K(\om)$ is in general not periodic. Intuitively, we would like to apply the deterministic result in a $\delta$-neighbourhood of each component of $K(\omega)$, since by assumption such neighbourhoods are pairwise disjoint. If we did it naively, however, then we could have for each component of $K(\omega)$ a different extension constant, since the components of $K(\omega)$ are balls with possibly different centres and radii from one another. Consequently, we would not be able to obtain uniform bounds for the extended function, which are crucial for equi-coerciveness.

The way we obtain uniform bounds relies on the following construction. Let us focus on a generic perforation $B(\theta(\om), r(\om))$, where $\theta(\om)$ and $r(\om)$ are the (random) centre and radius, and denote with $A_\delta(\om)$ the concentric (spherical) annulus of radii $r(\om)$ and $r(\om)+\delta$. The idea is to divide the hole into dyadic annuli such that the ratio between the outer and inner radii of each annulus is fixed, and depends only on $\delta$. Since the deterministic extension is invariant under translations and homotheties, we can apply it iteratively and construct an extension from $A_\delta(\om)$ (where the function is defined thanks to the $\delta$-separation) to $B(\theta(\om), r(\om)+\delta)$, in a way that controls the extension constant at every step (see Lemma \ref{t:ext-cell}). We then repeat this procedure for every inclusion, and obtain an extension result in $A \setminus \e K(\om)$, with an extension constant independent of $\e$ and of $\om$ (Theorem \ref{stochastic-extension-lemma-SBV}). This is a key ingredient in the proof of the compactness for sequences with bounded energies $E_\e(\om)$ (Proposition \ref{prop:comp}).

\medskip

\textbf{The $\Gamma$-convergence result.}
Once the compactness is established, we prove the stochastic $\Gamma$-convergence of $E_\e(\omega)$ for $\e\to 0$ (Theorems \ref{en-density_vs} and \ref{theo:stochMS}). Our strategy is to resort to a perturbation argument. Namely, we first introduce a perturbed functional $E^{k}_\e(\omega)$, with volume and surface densities given by $f^k:=a^k f$ and $g^k:=a^k g$, where 
\begin{equation*}
		a^{k}(\omega,x)
		:= \begin{cases}
		1 & \mbox{if} \ x \in \R^n \setminus K(\omega), \\
		\frac{1}{k} & \mbox{if} \ x \in K(\omega).
		\end{cases}
	\end{equation*}
In other words, $E^{k}_\e(\omega)$ is obtained from $E_\e(\om)$ by filling the holes with a coefficient $\frac1k$, with $k\in \N$. The perturbed functionals are non-degenerate and coercive, hence for fixed $k$ the $\Gamma$-limit of $E_\e^k$ for $\e\to 0$ exists almost surely by \cite[Theorem 3.12]{CDMSZ2}. Moreover, we can identify the limit volume and surface energy densities, which are given by
\begin{equation}\label{fk}
f^k_{\rm hom}(\om,\xi)
=\lim_{t \rightarrow +\infty} \frac{1}{t^n}\inf\left\{\int_{tQ}f^k\left(\omega,x,\nabla u\right)dx
: u\in W^{1,p}(tQ), \ u = \xi\cdot x \ \text{ near } \partial (tQ)\right\},
\end{equation}
and 
\begin{equation}\label{gk}
g^k_{\rm hom}(\om,\nu) =\lim_{t \rightarrow +\infty} \frac{1}{t^{n-1}}\inf\left\{
\int_{S_u\cap tQ^\nu}\!\!\!g^k\left(\omega,x,\nu_u\right)\, d\HH
: u\in \mathcal P(tQ^\nu), \ u = u_{0,1,\nu} \ \text{ near } \partial (t Q^\nu)\right\},
\end{equation}
where $\xi\in \R^n$, $\nu\in \mathbb{S}^{n-1}$, $Q^\nu$ is the rotated unit cube with one face perpendicular to $\nu$, $u_{0,1,\nu}$ is the piecewise constant function equal to $1$ in the upper semi-cube and $0$ in the lower semi-cube, and $\mathcal{P}$ denotes the set of partitions with values in $\{0,1\}$.

The volume and the surface densities $f_{\rm hom}$ and $g_{\rm hom}$ of the $\Gamma$-limit of $E_\e(\omega)$ are then obtained as the limits as $k\to +\infty$ of $f^k_{\textrm{hom}}$ and $g^k_{\textrm{hom}}$, respectively. The most delicate part in the proof is to show that these limits coincide with
\begin{equation}\label{fl}
\lim_{t \rightarrow +\infty} \frac{1}{t^n}\inf\left\{\int_{tQ\setminus K(\om)}f\left(\omega,x,\nabla u\right)dx
: u\in W^{1,p}(tQ), \ u = \xi\cdot x \ \text{ near } \partial (tQ)\right\}
\end{equation}
and 
\begin{equation}\label{gl}
\lim_{t \rightarrow +\infty} \frac{1}{t^{n-1}}\inf\left\{
\int_{S_u\cap (tQ^\nu\setminus K(\om))}\!\!\!g\left(\omega,x,\nu_u\right)\, d\HH
: u\in \mathcal P(tQ^\nu), \ u = u_{0,1,\nu} \ \text{ near } \partial (t Q^\nu)\right\},
\end{equation}
respectively. This step requires a careful use of extension techniques for Sobolev functions (Lemma \ref{Sob:ext-cell}) and for Caccioppoli partitions (Lemma \ref{Part:ext-cell}) separately, in order to construct, starting from a competitor for the minimisation problem in \eqref{fl} (resp.~\eqref{gl}) a competitor for the minimisation problem in \eqref{fk} (resp.~\eqref{gk}). 
Lemma \ref{Part:ext-cell}, in particular, requires the use of a technical lemma proved by Congedo Tamanini in \cite{TC1} (see also \cite{TC2}), which establishes some regularity properties for minimisers of the perimeter functional. These regularity properties, in turn, ensure that minimising partitions are constant on a sphere  around each perforation, from which we can then perform a trivial extension at no additional energetic cost.

Finally, our assumptions on the geometry of $K(\om)$ allow us to prove that the limit densities $f_{\rm hom}$ and $g_{\rm hom}$ are non-degenerate.

\subsection{Conclusions and outlook.} In this paper we prove a stochastic homogenisation result for free-discontinuity functionals on randomly perforated domains, without imposing any boundary conditions on the perforations. 
Our approach relies on the construction of an extension operator guaranteeing that, given a function in the perforated domain, the extended function in the whole domain is bounded, in energy, in terms of the original function. The construction of the extension operator, in turn, is guaranteed by our assumptions on the geometry of the randomly perforated domain. In particular, the assumption of $\delta$-separation of the holes is crucial in our analysis. This assumption, moreover, also ensures that the \emph{density} of the random domain is strictly positive, and hence the non-degeneracy of the limit energy. 

It would be interesting to investigate whether our result could work in the more general case where the existence of a fixed safety distance $\delta$ is replaced by a more global condition of ``average'' separation, \textit{e.g.} in the spirit of \cite[Section 8.4]{JiKoOl-book}.


\section{Setting of the problem and statement of the main result} 

\subsection{Notation}
We introduce here all the notation that we need.
\begin{itemize}
	\item $\N^\ast:= \{z\in \Z: z\geq 1\}$;
	\item For $\rho>0$ and $\theta\in \R^n$ we define $Q_{\rho}(\theta):=\{x \in \mathbb{R}^{n}  :  |x_{i}-\theta_i|<\frac{\rho}2, \ i=1,..,n\}$; we use the shorthands 
	$Q_{\rho} = Q_{\rho}(0)$ and $Q=Q_1$; 
	\item For $\rho>0$ and $\theta\in \R^n$ we define $B(\theta,\rho):=\{x\in \R^n: |x-\theta|<\rho\}$;
	\item For $0<r<s$ and $\theta\in \R^n$ we define the open annulus $B_{r,s}(\theta):=B(\theta,s)\setminus \overline B(\theta,r)$ and denote $B_{r,s}=B_{r,s}(0)$;
	\item $\Sph^{n-1}:=\{x\in \mathbb{R}^{n}: |x|=1\}$;
	\item $\mathcal{L}^n$ denotes the Lebesgue measure on $\R^n$ and $\mathcal{H}^{n-1}$ the $(n-1)$-dimensional Hausdorff measure on $\R^n$;
	\item $\mathcal{A}$ denotes the family of bounded domains of $\mathbb{R}^{n}$ with Lipschitz boundary; 
	\item We denote with $\B^n$ the Borel $\sigma$-algebra on $\R^n$ and with $\B(\Sph^{n-1})$ the Borel $\sigma$-algebra on $\mathbb{S}^{n-1}$;
	\item For $\xi \in \R^n$, we denote with $\ell_\xi$ the linear function $\ell_\xi(x) = \xi\cdot x$ for $x\in \R^n$;
	\item For $x\in \R^n$, $t>0$ and $\nu\in \Sph^{n-1}$, we denote with $Q^{\nu}_t(x)$ the cube of side-length $t>0$, centred at $x$ with one face orthogonal to $\nu$;
	\item For $x\in \R^n$ and $\nu\in \Sph^{n-1}$, we set 
\begin{equation*}
			u_{x,1,\nu}(y):= \begin{cases} 
			1 & \mbox{if} \ (y-x)\cdot \nu \geq 0 \\ 
			0 & \mbox{if} \ (y-x)\cdot \nu < 0.
			\end{cases}
		\end{equation*}
\end{itemize}

The functional setting for our analysis is that of \textit{generalised special functions of bounded variation}. We recall some basic definitions and refer to \cite{AFP} for a more comprehensive introduction to the topic.  

For $A\in \mathcal A$, the space of special functions of bounded variation in $A$ is defined as
\begin{equation*}
SBV(A)= \{u \in BV(A): \ Du = \nabla u \,\mathcal{L}^{n} + (u^{+} - u^{-})\nu_{u} \mathcal{H}^{n-1} \mres S_{u}\}.
\end{equation*}
Here $S_{u}$ denotes the approximate discontinuity set of $u$, $\nu_{u}$ is the generalised normal to $S_{u}$, $u^{+}$ and $u^{-}$ are the traces of $u$ on both sides of $S_{u}$. We also consider the space 
$$
\mathcal{P}(A)=\{u\in SBV(A): \nabla u=0, \; u\in \{0,1\} \ \mathcal{L}^{n}\textrm{-a.e.}, \,\mathcal{H}^{n-1}(S_u)<+\infty\};
$$
hence every $u$ in $\mathcal{P}(A)$ is a partition in the sense of \cite[Definition 4.21]{AFP}. 

For $p>1$, we define the following vector subspace of $SBV(A)$:
\begin{equation*}
SBV^{p}(A)= \{u \in SBV(A): \ \nabla u\in L^{p}(A) \text{ and } \mathcal{H}^{n-1}(S_{u})<+\infty\}.
\end{equation*}
We consider also the larger space of \textit{generalised special functions of bounded variation} in $A$, 
\begin{equation*}
GSBV(A)= \{u \in L^{1}(A): \ (u \wedge m) \vee (-m) \in SBV(A) \text{ for all } m\in \N\}.
\end{equation*}
By analogy with the case of $SBV$ functions, we write
\begin{equation*}
GSBV^{p}(A)= \{u \in GSBV(A): \ \nabla u\in L^{p}(A) \text{ and } \mathcal{H}^{n-1}(S_{u})<+\infty\}.
\end{equation*}

\subsection{Volume and surface integrands} Let $p>1$, $0<c_1\leq c_2 <+\infty$, $L>0$, and  
let $f \colon \R^n \times \R^n \longrightarrow [0,+\infty)$ be a Borel function on $\R^n{\times} \R^{n}$ satisfying the following conditions: 
\begin{itemize}
\item[$(f1)$] (lower bound) for every $x \in \R^n$ and every $\xi \in \R^{n}$
$$
c_1 |\xi|^p \leq f(x,\xi);
$$
\item[$(f2)$] (upper bound) for every $x \in \R^n$ and every $\xi \in \R^{n}$
$$
f(x,\xi) \leq c_2(1+|\xi|^p);
$$
\item[$(f3)$] (continuity in $\xi$) for every $x \in \R^n$ we have
$$
|f(x,\xi_1)-f(x,\xi_2)| \leq L\big(1+|\xi_1|^{p-1}+|\xi_2|^{p-1}\big)|\xi_1-\xi_2|
$$
for every $\xi_1$, $\xi_2 \in \R^{n}$.
\end{itemize}
Let $0<c_3\leq c_4 <+\infty$ and  let $g \colon \R^n \times \Sph^{n-1} \longrightarrow [0,+\infty)$ be a Borel function on $\R^n \times \Sph^{n-1}$ satisfying the following conditions: 
\begin{itemize}
\item[$(g1)$] (lower bound) for every $x \in \R^n$ and every $\nu \in \Sph^{n-1}$
$$
c_3 \leq g(x,\nu);
$$
\item[$(g2)$] (upper bound) for every $x  \in \R^n$ and every $\nu \in \Sph^{n-1}$
$$
g(x,\nu) \leq c_4;
$$
\item[$(g3)$] (symmetry) for every $x \in \R^n$  and every $\nu \in \Sph^{n-1}$
$$
g(x,\nu)=g(x,-\nu).
$$
\end{itemize}

\subsection{Stochastic framework} Let $(\Omega,\mathcal{T},{P})$ be a complete probability space. 
We start by recalling some definitions. 

\begin{defi}[Group of $P$-preserving transformations]
A group of $P$-preserving transformations on $(\Om,\T,P)$ is a family $(\tau_y)_{y\in \R^n}$ of $\T$-measurable mappings $\tau_y \colon \Om \to \Om$ satisfying the following properties:
\begin{itemize}
\item (measurability) the map $(\om,y) \mapsto \tau_y(\om)$ is $(\T\otimes \B^n, \T)$-measurable;
\item (bijectivity) $\tau_y$ is bijective for every $y\in \R^n$;
\item (invariance) $P(\tau_y(E))=P(E)$,\, for every $E\in \T$ and every $y\in \R^n$;
\item (group property) $\tau_{0}={\mathrm id}_\Om$ (the identity map on $\Om$) and $\tau_{y+y'}=\tau_{y} \circ \tau_{y'}$ for every $y, y'\in \R^n$. 
\end{itemize}
\noindent If, in addition, every set $E\in \T$ which satisfies 
$\tau_y (E)=E$ for every $y\in \R^n$
has probability $0$ or $1$, then $(\tau_y)_{y\in\R^n}$ is called {\it ergodic}.
\end{defi}

We are now in a position to define the notion of stationary random integrand.
\begin{defi}[Stationary random integrand]\label{def:stationary-integrand}  Let $(\tau_y)_{y\in \R^n}$ be a group of $P$-preserving transformations on $(\Om,\T,P)$. 
We say that $f\colon \Omega \times \R^n \times \R^n \longrightarrow [0,+\infty)$ is a \emph{stationary random volume integrand} if
\begin{itemize}
\item[(a)]  $f$ is $(\T\otimes \B^n\otimes \B^{n})$-measurable; 

\smallskip

\item[(b)]  $f(\om,\cdot,\cdot)$ satisfies $(f1)$--$(f3)$ for every $\om\in \Om$;

\smallskip
 
\item[(c)] $f(\om,x+y,\xi)=f(\tau_y(\om),x,\xi)$,
for every $\om\in \Om$, $x,y\in \R^n$, and $\xi\in \R^{n}$.
\end{itemize}

\smallskip

\noindent Similarly, we say that $g \colon \Om \times \R^n \times \Sph^{n-1} \longrightarrow [0,+\infty)$ is a \textit{stationary random surface integrand}  if
\begin{itemize} 
\item[(d)] $g$ is $(\T\otimes \B^n\otimes \B(\Sph^{n-1}))$-measurable;

\smallskip

\item[(e)] $g(\om,\cdot,\cdot)$ satisfies $(g1)$--$(g3)$ for every $\om\in \Om$;

\smallskip

\item[(f)] $g(\om,x+y,\nu)=g(\tau_y(\om),x,\nu)$,
for every $\om\in \Om$, $x,y\in \R^n$, and $\nu \in \Sph^{n-1}$.
\end{itemize}

\noindent If in addition $(\tau_y)_{y\in \R^n}$ is an ergodic group of $P$-preserving transformations, then we say that $f$ and $g$ are ergodic.  

\end{defi} 

We also recall the definition of random domain. The main difference with the classical definition given in, \textit{e.g.}, \cite[Chapter 8]{JiKoOl-book} is that we do not assume any ergodicity for the group $(\tau_y)_{y\in \R^n}$. 

\begin{defi}[Random domain] Let $(\tau_y)_{y\in \R^n}$ be a group of $P$-preserving transformations on $(\Om,\T,P)$. 
A \emph{random domain} is a map $\om \mapsto D(\om)$ from $\Om$ to $\R^n$ such that: 
\begin{itemize}
\item the map $(\om,x) \mapsto \chi_{D(\om)}(x)$ is $(\T\otimes \B^n)$-measurable;
\item for every $\om\in \Om$, $x,y\in \R^n$ it holds
\begin{equation}\label{sta-dom}
\chi_{D(\om)}(x+y)=\chi_{D(\tau_y\om)}(x).
\end{equation}
 \end{itemize}
\end{defi}

\begin{rem}
We note that $D$ is a random domain if and only if for every $\om \in \Omega$
\begin{equation}\label{D-1}
D(\om)=\big\{x\in \R^n \colon \tau_x \om \in \widetilde D\big\}
\end{equation}
for some $\widetilde D \in \T$. 
Indeed if $D(\om)$ is as in \eqref{D-1} for some $\widetilde D\in \T$, then it is immediate to check that $\chi_{D(\om)}$ satisfies \eqref{sta-dom}. If on the other hand $\chi_{D(\om)}$ satisfies \eqref{sta-dom}, we have that $\chi_{D(\om)}(x)=\chi_{D(\tau_x \om)}(0)$ and therefore
\begin{equation*}
\widetilde D = \{\om \in \Om \colon 0\in D(\om)\}. 
\end{equation*}

\end{rem}

\begin{defi}[Density of a random domain]\label{dens-sd} Let $D$ be a random domain, let $\widetilde D \in \T$ be as in \eqref{D-1}, and let $\mathscr I\subset \T$ denote the $\sigma$-algebra of $(\tau_y)_{y\in \R^n}$-invariant sets; that is, $\mathscr I:= \{E\in \T: \tau_y(E) = E \quad \forall\, y\in \R^n\}$.
The function $d \colon \Om \to [0,+\infty)$ defined for every $\om\in \Om$ as $d(\om):=\mathbb{E}[\chi_{\widetilde D}|\mathscr I](\om)$ is called the pointwise density of $D$. 
\end{defi}

\begin{rem}
By the definition of conditional expectation we have that
$$
\int_{\Om} \mathbb{E}[\chi_{\widetilde D}|\mathscr I](\om)\, dP(\om) = \int_{\Om} \chi_{\widetilde D}(\om)\, dP(\om) = P(\widetilde D),
$$
since $\Om\in \mathscr I$. The nonnegative number $\bar d:=P(\widetilde D)$ is usually referred to as the (average) density of $D(\om)$ (see e.g. \cite[Chapter 8]{JiKoOl-book}).   
\end{rem}

\begin{rem}[Birkhoff's Ergodic Theorem]\label{rem:B}
Let $D$ be a random domain and let $\e>0$. For every $\om \in \Om$ and $x\in \R^n$ set 
\begin{equation}\label{d-e}
d_\e(\om,x):=\chi_{\e D(\om)}(x).
\end{equation}
Then, the Birkhoff Ergodic Theorem ensures that for $P$-a.e. $\om \in \Om$
\begin{equation*}
d_\e(\om, \cdot) \stackrel{*} {\wto} d(\om):=\mathbb{E}[\chi_{\widetilde D}|\mathscr I](\om) \quad \text{in} \quad L^\infty_{\rm loc}(\R^n)
\end{equation*}
as $\e\to 0$. If moreover $D$ is ergodic then 
$$
d_\e(\om, \cdot) \stackrel{*} {\wto} \bar d:=P(\widetilde D) \quad \text{in} \quad L^\infty_{\rm loc}(\R^n). 
$$
\end{rem}

\noindent We require the following additional assumptions on the geometry of the random domain $D$.

\begin{defi}[Random perforated domain]\label{ass:random-domain}
Let $\delta>0$ and $r_\ast>0$ be fixed and independent of $\om$, let $D$ be a random domain, and set, for $\om\in \Omega$,  
$K(\om):=\R^n\setminus D(\om)$. We say that $K$ is a random perforated domain if:

\smallskip

(K1) for every $\om \in \Om$ the set $K(\om)$ is the union of closed balls with radius smaller than $r_*$;

\smallskip

(K2) for every $\om \in \Om$ the distance between any two distinct balls in $K(\om)$ is larger that $2\delta$.

\end{defi}

Properties (K1) and (K2) can be rephrased as follows: 

\begin{itemize}
\item $K(\om)$ is of the form
\begin{equation}
		K(\omega) := \bigcup_{i \in \mathcal{I}} \overline B(\theta_{i}(\omega),r_i(\omega)),
		\label{def:P}
	\end{equation}
with $r_i(\omega)\in (0,r_\ast)$, $\theta_i(\om)\in \R^n$, and with $\theta_i(\om)\neq \theta_j(\om)$ for $i\neq j$, for every $i,j\in \mathcal I$ and for $P$-a.e.~$\omega\in \Om$;

\item for every $i,j\in \mathcal I$ with $i\neq j$ 
	\begin{equation}\label{separation-K}
		B(\theta_{i}(\omega),r_i(\omega)+\delta) \cap B(\theta_{j}(\omega),r_j(\omega)+\delta)=\emptyset.
	\end{equation}

\end{itemize}


\noindent 
The set $K(\omega)$ is a special type of \textit{random spherical structure}, as defined in \cite[Definition 8.19]{JiKoOl-book}. It is special because of the strong $2\delta$-separation of the spherical perforations, which is crucial in our analysis. 

\begin{rem}[Example of a random perforated domain]
The simplest example of a random perforated domain can be obtained as follows. Let $\mathcal L\subset \R^n$ be a regular Bravais lattice (\textit{e.g.}, the cubic lattice or the triangular lattice for $n=2$).  Let $Q(\mathcal L)$ denote the periodicity cell of the lattice, and let $B\subset \subset Q(\mathcal L)$ be a ball well contained in the cell. Then an admissible set of perforations is
given by
$$
K_\mathcal L(\om) = \bigcup_{y\in \mathcal Y(\om)}(B+y),
$$ 
where $\mathcal Y(\om) \subset \mathcal L$ is a random set obtained, for instance, by running i.i.d.\ Bernoulli trials at each $y\in \mathcal L$. 
\end{rem}

We now show that a random stationary domain as in Definition \ref{ass:random-domain} has a positive pointwise density $d(\om)$ for $P$-a.e. $\om\in\Om$.

\begin{prop}\label{property:d}
Let $D$ be a random perforated domain as in Definition \ref{ass:random-domain} and let $d$ be its pointwise density as in Definition \ref{dens-sd}. Then $d(\om)>0$ for $P$-a.e. $\om\in\Om$.
\end{prop}

\begin{proof}
Let $\e>0$ be small and let $d_\e$ be as in \eqref{d-e}; then for $\om \in \Om$
$$
\int_Q d_\e(\om,x)\dx = \int_Q \chi_{\e D(\om)}(x)\dx = \mathcal L^n(Q\cap \e D(\om))=\mathcal L^n(Q\setminus \e K(\om)),
$$
where $Q$ denotes the unit cube. By the Birkhoff Ergodic Theorem we deduce that, in particular, 
\begin{equation}\label{conv-leb-meas}
\lim_{\e \to 0}\mathcal L^n(Q\setminus \e K(\om))= d(\om), 
\end{equation}
for $P$-a.e. $\om \in \Om$. We now show that, by Definition \ref{ass:random-domain}, the left hand-side of \eqref{conv-leb-meas} can be estimated from below by a positive constant independent of $\om$. 

Let $N_\e$ denote the number of components $ \e B(\theta_{i}(\omega),r_i(\omega))$ of $\e K(\om)$ such that  $\e B(\theta_{i}(\omega),r_i(\omega)+\delta)$ is contained in $Q$. Note that the total number of perforations intersecting $Q$ is $N_\e + N^b_\e$, where $N_\e^b$ denotes the number of ``boundary'' perforations, namely the components of $\e K(\om)$ that do not intersect $Q_{1-2\e(r_\ast + \delta)}$. We can neglect the boundary perforations in the estimate of $\mathcal L^n(Q\cap \e K(\om))$ (and hence of $\mathcal L^n(Q\setminus \e K(\om))$) since they provide an infinitesimal volume contribution. The assumption of $2\delta \e$-separation of the components of $\e K(\om)$ ensures that $N_\e \leq (2\e \delta)^{-n}$. 

If $N_\e \ll \e ^{-n}$ we immediately get
$$
\mathcal L^n(Q\cap \e K(\om)) \leq c_n N_\e \e^n r_*^n \to 0 \quad \text{as} \quad \e \to 0,
$$
where $c_n := \mathcal{L}^n(B(0,1))$ and therefore
$$
\mathcal L^n(Q\setminus \e K(\om))= \mathcal L^n(Q)- \mathcal L^n(Q\cap \e K(\om)) \geq \frac{1}{2} 
$$
for small enough $\e>0$ and every $\om \in \Om$. 

We now assume that $N_\e \sim \e^{-n}$. First of all, by the definition of $N_\e$, and by the $2\delta \e$-separation of the components of $\e K(\om)$, we have that 
$$
\e B(\theta_j(\om),r_j(\om)+\delta) \setminus \e B(\theta_j(\om),r_j(\om)) \subset Q\setminus \e K(\om). 
$$
Consequently we have 
$$
 \mathcal L^n(Q\setminus \e K(\om)) \geq N_\e c_n(\e \delta)^n,
$$
where to establish the last inequality we have used that  
$$
\mathcal L^n (\e B(\theta_j(\om),r_j(\om)+\delta) \setminus \e B(\theta_j(\om),r_j(\om)))= \e^n c_n\big((r_j(\om)+\delta)^n-r_j(\om)^n\big)\geq c_n\e^n \delta^n. 
$$
Therefore also in this case we have that
$$
\lim_{\e\to 0} \mathcal L^n(Q\setminus \e K(\om)) \geq \lim_{\e \to 0} N_\e c_n(\e \delta)^n =c>0,
$$
and this concludes the proof.
\end{proof}

\subsection{Energy functionals and statement of the main result} We now introduce the sequence of functionals we are going to study.

For $\om\in \Om$ and $\e>0$ we consider the random functionals $E_\e(\om)\colon L^1_{\rm{loc}}(\R^n) \times \mathcal{A} \longrightarrow [0,+\infty]$ defined as 
\begin{equation}\label{def:Eneps}
		E_\e(\omega)(u,A)
		\displaystyle
		:= \begin{cases}
			\displaystyle
			\int_{A\setminus \e K(\omega)}\!\!\!f\left(\omega,\frac{x}{\e},\nabla u\right)dx + \int_{S_u\cap (A\setminus \e K(\omega))}\!\!\!g\left(\omega,\frac{x}{\e},\nu\right)\, d\HH & \mbox{if} \ u_{|A} \in GSBV^p(A), \\
			+\infty & \mbox{otherwise,}
		\end{cases}
	\end{equation}
where $f$ and $g$ are stationary random integrands as in Definition \ref{def:stationary-integrand}, and $K(\om)$ is as in Definition \ref{ass:random-domain} (see Figure \ref{f:RPD}).

\begin{figure}[h!]
\includegraphics[width=0.38\textwidth]{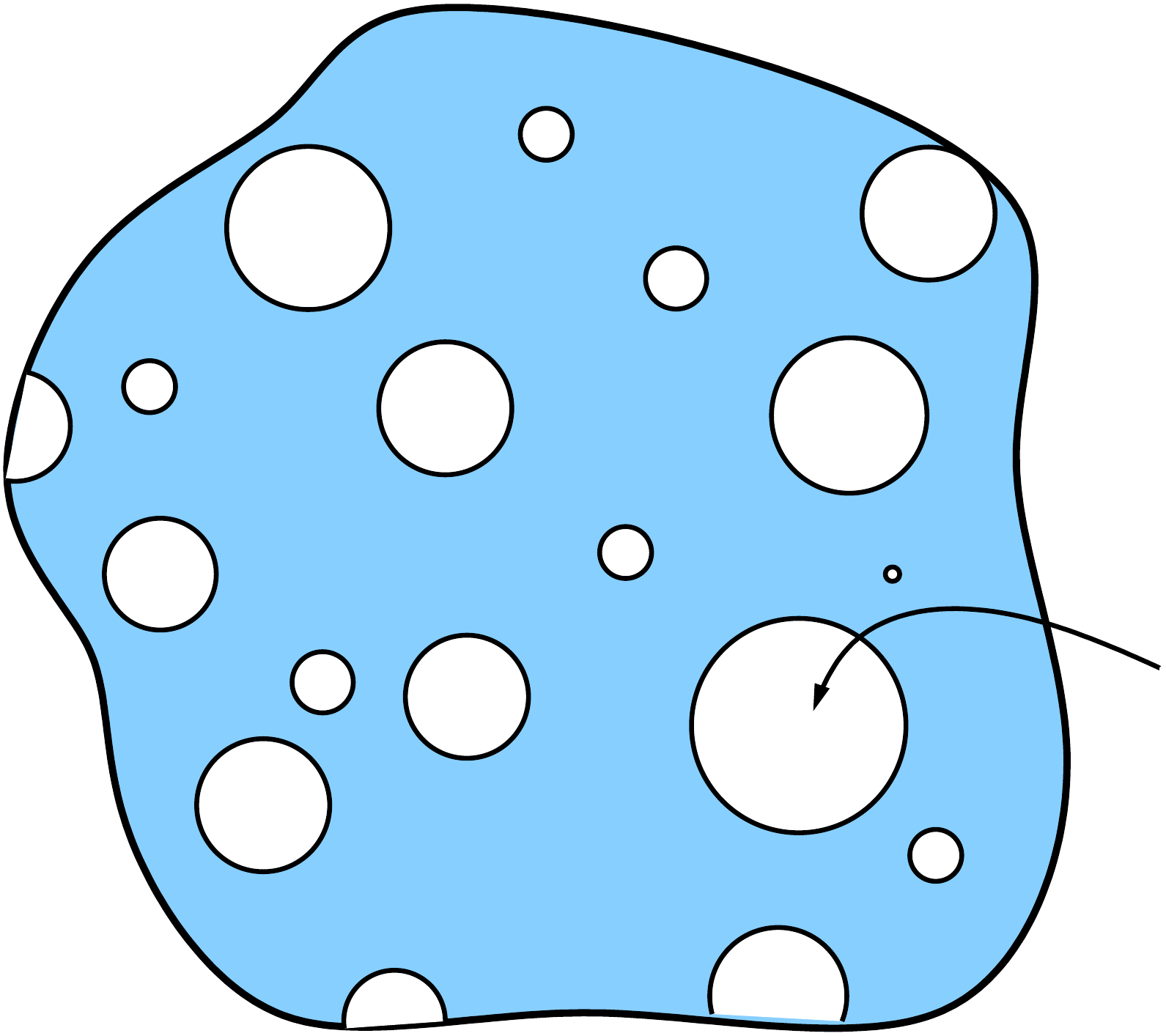}
\caption{The random perforated domain $A\setminus \e K(\om)$.}
\label{f:RPD}
\begin{picture}(0,0) 
\put(75,50){\Large{$A$}}
\put(88,82){{$\e B(\theta_i(\om),r_i(\om))$}}
\end{picture}
\end{figure}

Let moreover $F(\om), G(\om): L^1_{\rm{loc}}(\R^n) \times \mathcal{A} \longrightarrow [0,+\infty]$ be defined as
\begin{equation}\label{Effe}
		F(\omega)(u,A)
		\displaystyle
		:= \begin{cases}
			\displaystyle
			\int_{A\setminus K(\omega)}f\left(\omega,x,\nabla u\right)dx & \mbox{if} \ u_{|A} \in W^{1,p}(A), \\
			+\infty & \mbox{otherwise,}
		\end{cases}
	\end{equation}
	and
	\begin{equation}\label{Gi}
G(\omega)(u,A)
		\displaystyle
		:= \begin{cases}
			\displaystyle
			\int_{S_u\cap (A\setminus K(\omega))}g\left(\omega,x,\nu_u\right)\, d\HH & \mbox{if} \ u_{|A} \in GSBV^p(A), \\
			+\infty & \mbox{otherwise.}
		\end{cases}
			\end{equation}
Let $A\in \mathcal{A}$ be fixed; for $v\in L^1_{\rm{loc}}(\R^n)$, with $v_{|A}\in W^{1,p}(A)$, we define
	\begin{equation}\label{emme}
		m^{1,p}_{F(\omega)}(v,A) := \inf\big\{F(\omega)(u,A): \ u\in L^1_{\rm{loc}}(\R^n), u_{|A}\in W^{1,p}(A), \ u = v \ \text{ near } \partial A\big\}.
	\end{equation}
Similarly, for $v\in L^1_{\rm{loc}}(\R^n)$, with $v_{|A} \in \mathcal P(A)$, we define
	\begin{equation}\label{emmeG}
		m^{\textrm{pc}}_{G(\omega)}(v,A) := \inf\big\{G(\omega)(u,A): \ u\in L^1_{\rm{loc}}(\R^n), u_{|A}\in \mathcal P(A), \ u = v \ \text{ near } \partial A\big\}.
	\end{equation}
In the formulas above, by ``$u=v$ near $\partial A$'' we mean that there exists a neighbourhood $U$ of $\partial A$ in $\R^n$ such that $u=v$ $\mathcal{L}^n$-a.e.\ in $U \cap A$.

\smallskip

The following theorem is the main result of this paper. 

\begin{theo}[Homogenisation theorem]\label{statement-main-thm}
Let $f$ and $g$ be stationary random volume and surface integrands, and let $D\subset \R^n$ be a random perforated domain as in Definition \ref{ass:random-domain}. Assume that the stationarity of $f$, $g$ and $D$ is satisfied with respect to the same group $(\tau_y)_{y\in \R^n}$ of $P$-preserving transformations on $(\Om,\T,P)$. Let $\e>0$, and let $E_\e$ be the functionals defined as in \eqref{def:Eneps}. 

\smallskip

I) {\rm (Compactness)} Let $\omega\in \Omega$ and $A\in \mathcal A$ be fixed; let $(u_\e) \subset L^1_{\rm loc}(\R^n)$ be such that 
$$
\sup_{\e>0} \Big(E_\e(\omega)(u_\e,A)+\|u_\e\|_{L^\infty (A\setminus \e K(\omega))}\Big)<+\infty.
$$
Then there exist a sequence $(\tilde u_\e) \subset SBV^p(A)\cap L^1_{\rm loc}(\R^n)$ and a function $u\in SBV^p(A)\cap L^1_{\rm loc}(\R^n)$ such that $\tilde u_\e = u_\e$ $\mathcal L^n$-a.e. in $A\setminus \e K(\omega)$ and (up to a subsequence not relabelled) $\tilde u_\e \to u$ strongly in $L^1(A)$. 

\smallskip

II) {\rm(Almost sure $\Gamma$-convergence)} There exists $\Om'\in \mathcal T$, with $P(\Om')=1$, such that for every $\om\in \Om'$ the functionals $E_\e(\om)$ $\Gamma$-converge with respect to the $L^1_{\rm loc}(\R^n)$-convergence, as $\e\to 0$, to the functional $E_{\rm hom}(\om) \colon  L^{1}_{\rm loc}(\mathbb{R}^{n})\times \mathcal{A} \longrightarrow [0,+\infty]$ given by
	\begin{equation}\label{aaa}
		E_{\rm hom}(u,A)
		\displaystyle
		= \begin{cases}
		\displaystyle
		\int_{A}f_{\rm hom}(\om,\nabla u) \ dx 
		+\int_{A\cap S_{u}}g_{\rm hom}(\om,\nu_u)\ d\mathcal{H}^{n-1} 
		& \mbox{if} \ u_{|_A} \in GSBV^{p}(A), \\
		+\infty & \mbox{otherwise}.
		\end{cases}
	\end{equation} 
In \eqref{aaa}, for every $\om\in \Omega'$, $\xi \in \mathbb{R}^{n}$, and $\nu \in \mathbb{S}^{n-1}$, 
\begin{equation*}
f_{\mathrm{hom}}(\om,\xi):= \lim_{t\to +\infty} \frac{1}{t^n} m^{1,p}_{F(\om)} (\ell_\xi, Q_t(0)),
\end{equation*}
and
\begin{equation*}
g_{\mathrm{hom}}(\om,\nu):=\lim_{t\to +\infty} \frac{1}{t^{n-1}}m^{\mathrm{pc}}_{G(\om)}(u_{0,1,\nu},Q^\nu_t(0)),
\end{equation*}
with $m^{1,p}_{F(\om)}$ and $m^{\mathrm{pc}}_{G(\om)}$ defined as in \eqref{emme} and \eqref{emmeG}, respectively. If, in addition, $f$, $g$ and $D$ are ergodic, then $f_{\mathrm{hom}}$ and $g_{\mathrm{hom}}$ are independent of $\om$. 

\smallskip

III) {\rm (Properties of $f_{\rm hom}$ and $g_{\rm hom}$)} The homogenised volume integrand $f_{\rm hom}$ satisfies the following properties:
\begin{itemize}
\item[i.] {\rm(measurability)} $f_{\rm hom}$ is $(\T\otimes \B^n)$-measurable;
\item[ii.] {\rm(bounds)} there exists $\tilde c_0>0$ such that 
$$
\tilde c_0 |\xi|^p\leq f_{\mathrm{hom}}(\om,\xi) \leq c_2(1+ |\xi|^p),
$$
for every $\om \in \Om'$ and every $\xi\in \R^n$, with $c_2$ as in $(f2)$;
\item[iii.] {\rm(continuity)} there exists $L'>0$ such that
$$
|f_{\rm hom}(\om,\xi_1)-f_{\rm hom}(\om,\xi_2)| \leq L'\big(1+|\xi_1|^{p-1}+|\xi_2|^{p-1}\big)|\xi_1-\xi_2|,
$$
for every $\om\in \Om'$ and every $\xi_1$, $\xi_2 \in \R^{n}$.
\end{itemize}
Additionally, the homogenised surface integrand $g_{\rm hom}$ satisfies:
\begin{itemize}
\item[iv.] {\rm(measurability)} $g_{\rm hom}$ is $(\T\otimes \B(\Sph^{n-1}))$-measurable;
\item[v.] {\rm(bounds)} there exists $\tilde c_0>0$ such that 
\begin{equation}\label{g-bds}
\tilde c_0 \leq g_{\mathrm{hom}}(\om,\nu) \leq  c_4,
\end{equation}
for every $\om\in \Om'$ and every $\nu\in \Sph^{n-1}$, with $c_4$ as in $(g2)$;
\item[vi.] {\rm(symmetry)} $g_{\rm hom}(\om,\nu)=g_{\rm hom}(\om,-\nu)$, for every $\om\in \Om'$ and every $\nu \in \Sph^{n-1}$.
\end{itemize}
\end{theo}
The proof of Theorem \ref{statement-main-thm} will be broken up into three main steps which will be, respectively, the object of Proposition \ref{prop:comp}, Theorem \ref{en-density_vs}, and Theorem \ref{theo:stochMS} below.

\section{Preliminaries} 
\noindent In this short section we collect two known results which will be used in what follows. 
The first one, Theorem \ref{estfixdomintro} is an extension result for $GSBV$-functions. The second result, Lemma \ref{fracture-lemma}, is a regularity result for minimal partitions.

For $p>1$ and $n\geq 2$ we introduce the shorthand $MS^p$ for the $p$-Mumford-Shah functional, namely we write
$$
MS^p(u,A):= \int_{A}|\nabla u|^p dx + \HH(S_u\cap A),
$$
where $A\in \mathcal{A}$ and $u\in GSBV^p(A)$. Moreover, if $u\in GSBV^p(B)$, with $B\in \mathcal{A}$ and $\bar A\subset B$, we use the notation
$$
MS^p(u,\bar A):= MS^p(u,A) +  \HH(S_u\cap \partial A).
$$
We now recall \cite[Theorem 1.1]{CS}. 

\begin{theo}\label{estfixdomintro}
Let $p>1$, let $A,A' \subset \mathbb{R}^n$ be bounded open sets with Lipschitz boundary 
and assume that $A'$ is connected, $A' \subset A$ and $A\setminus A' \subset\subset A$.
Then there exists an extension operator 
$T: GSBV^p(A') \longrightarrow GSBV^p(A)$
and a constant $c=c(n,p,A,A') > 0$ such that

\begin{itemize}
\item $T u = u \quad \mathcal{L}^n\text{-a.e. in} \; A'$, 

\smallskip

\item $MS^p(T u,A) \leq c MS^p(u,A')$,
\end{itemize}
for every $u \in GSBV^p(A')$.
The constant $c$ is invariant under translations and homotheties. 

If in addition $u\in L^\infty(A')$, then $Tu \in SBV^p(A)\cap L^\infty(A)$, and $\|Tu\|_{L^\infty(A)} = \|u\|_{L^\infty(A')}$.
\end{theo}

\begin{rem}
The result in \cite{CS} is stated and proven in $GSBV^p$ for the most classical case $p=2$, but the general case of $GSBV^p$ for $p>1$ follows immediately. In fact, a key 
tool of the proof in \cite{CS} is the density lower bound proved in \cite{DeGiorgi} (see also \cite{DMMS92}), which is actually valid for any $p>1$ (see for instance \cite[Theorem 7.21]{AFP}). 

\end{rem}

We now state a technical lemma (see \cite[Lemma 4.5]{TC1}, and see also \cite[Lemma 2.5]{TC2} for a more general version of the result) for (locally) minimal partitions.

\begin{lemma}
	Let $n \geqslant 2$ and $\tau \in (0,1]$ be fixed. There exists a constant $\gamma=\gamma(n,\tau)>0$ such that if $0<s\leq r$, and $u\in \mathcal{P}(B_{r,r+s})$ verifies the following hypotheses: 
	\begin{enumerate}
		\item[$(H1)$]$\displaystyle\mathcal{H}^{n-1}(S_{u} \cap B_{r,r+s}) 
		\leq \mathcal{H}^{n-1}(S_{v} \cap B_{r,r+s})$ 
		for every competitor $v \in \mathcal{P}(B_{r,r+s})$ satisfying  $\supp(u-v) \subset B_{r,r+s}$; 
		\smallskip
		\item[$(H2)$]$\displaystyle \mathcal{H}^{n-1}(S_{u} \cap B_{r,r+s}) \leq \gamma s^{n-1}$; 
	\end{enumerate}
	then for every $r_{0}$ and $s_{0}$ such that $r \leq r_{0} < r_{0} + s_{0} \leq r+s$ and $s_{0} \geq \tau s$, there exists a radius $\bar{r} \in (r + s_{0}/3, r_{0} + 2 s_{0}/3)$ with the property that
	\begin{equation*}
	S_{u} \cap \partial B_{\bar{r}} = \emptyset.
	\end{equation*}
	\label{fracture-lemma}
\end{lemma}


\section{Extension results and compactness}
\noindent In this section we prove a compactness result for sequences $(u_\e)$ satisfying the bound 
\begin{equation}\label{venerdi}
E_\e(\om)(u_\e,A) +\|u_\e\|_{L^\infty(A)}\leq C \; \text{ for every}\; \e>0,
\end{equation}
for a constant $C>0$ independent of $\e>0$, where $A\in \mathcal{A}$ and $E_\e$ is defined as in \eqref{def:Eneps}, and $\om \in \Om$.

\smallskip

By definition of the functionals $E_\e(\om)$, the bound in \eqref{venerdi} does not provide any information on the $BV$-norm of $u_\e$ in $A\cap \e K(\om)$. To gain the desired bound, we show that $(u_\e)$ can be actually replaced by a sequence $(\tilde u_\e)\subset SBV^p(A)$ satisfying the two following properties: 
\begin{equation}\label{two-prop}
\tilde u_\e \equiv u_\e \quad \text{in}\; A\setminus \e K(\om) \quad \text{and}\quad  \sup_{\e}\|\tilde u_\e\|_{BV(A)}<+\infty.
\end{equation} 
In particular, $\tilde u_\e$ is energetically equivalent to $u_\e$. 
To prove the existence of such a sequence, we resort to a new extension result for functions defined on a perforated domain without assuming any periodicity on the distribution of the perforations (cf. Cagnetti and Scardia \cite{CS} for the case of periodically distributed perforations). 

\subsection{Extension}\label{sub:ext} The main result of this subsection is a $GSBV$-extension result from $A\setminus \e K(\om)$ to $A$ (cf. Theorem \ref{stochastic-extension-lemma-SBV}). 
Since this result is proven for $\om\in \Om$ fixed, in what follows we omit the dependence of the set $K(\om)$ on the random parameter $\om$. 
Hence below $K$ denotes any subset of $\R^n$ satisfying the two properties \eqref{def:P} and \eqref{separation-K} (cf. Definition \ref{ass:random-domain}).
	
Loosely speaking, to prove the desired $GSBV$-extension result we would like to apply Theorem \ref{estfixdomintro} in a $\delta$-neighbourhood of each component $B(\theta_{i},r_i)$ of $K$ (which are pairwise disjoint by assumption \eqref{separation-K}). If we did it naively, however, we could have for each $B(\theta_{i},r_i)$ a different extension constant. 
Lemma \ref{t:ext-cell} below ensures that the extension constant can be actually taken to be independent of $\theta_{i}$ and $r_i$.

\begin{lemma}[$GSBV$-extension in an annulus]\label{t:ext-cell}
Let $n\geq 2$ and $p>1$; let $\delta >0$ and $r_\ast>\delta$ be fixed. Let $\theta\in \R^n$ and $0<r<r_\ast$; then there exist an extension operator $T_{\theta, r} : GSBV^p(B_{r,r+\delta}(\theta)) \rightarrow GSBV^p(B(\theta,r+\delta))$ and a constant $c=c(n,p,\delta, r_\ast)>0$ such that
	\begin{equation*}
	\begin{split}
	T_{\theta,r} u = u \ \mathcal{L}^n\text{-a.e. in } \ B_{r,r+\delta}(\theta),
	\\
	MS^p(T_{\theta,r} u, B(\theta,r+\delta))  \leq c \, MS^p\big(u, B_{r,r+\delta}(\theta)\big)
	\end{split}
		\end{equation*}
for every $u \in GSBV^p(B_{r,r+\delta}(\theta))$. The constant $c$ is invariant under translations and homotheties. If in addition $u\in L^\infty(B_{r,r+\delta}(\theta))$, then $T_{\theta,r}u \in  SBV^p(B(\theta,r+\delta))\cap L^\infty(B(\theta,r+\delta))$, and $\|T_{\theta,r}u\|_{L^\infty(B(\theta,r+\delta))} = \|u\|_{L^\infty(B_{r,r+\delta}(\theta))}$.
\end{lemma}

\begin{proof}
Let $u\in GSBV^p(B_{r,r+\delta}(\theta))$. We treat the cases $r<\delta$ and $r\geq \delta$ separately. 

\medskip

\noindent
\textbf{Case 1: $r<\delta$.} Note that in this case $B_{r,2r}(\theta)\subset B_{r,r+\delta}(\theta)$.

Let $v:=u_{|_{B_{r,2r}(\theta)}}$. By applying Theorem \ref{estfixdomintro} with $A'=B_{r,2r}(\theta)$ and $A=B(\theta,2r)$, we deduce the existence of a constant $c=c(n,p)>0$ (independent of $\theta$ and $r$) and a function $w \in GSBV^p(B(\theta,2r))$ satisfying $
w = v =u$ $\mathcal L^n$-a.e. in $B_{r,2r}(\theta)$ and		
\begin{equation}\label{s:MS-2r-2}
MS^p(w, B(\theta,2r))  \leq c\, MS^p(v, B_{r,2r}(\theta))=c\, MS^p(u, B_{r,2r}(\theta)) \leq c\, MS^p(u, B_{r,r+\delta}(\theta)).
\end{equation}
We now define the function $\tilde u$ in $B(\theta,r+\delta)$ as follows:

\begin{equation}\label{ext:1}
		\tilde u :=
		\begin{cases}
			u &\mbox{ in } B_{r,r+\delta}(\theta), \\
			w_{|\overline B(\theta,r)} &\mbox{ in } \overline B(\theta,r).
		\end{cases}
	\end{equation}
Clearly, $\tilde u\in GSBV^p(B(\theta,r+\delta))$, $\tilde u = u$ in $B_{r,r+\delta}(\theta)$, and 
	\begin{align*}
		MS^p(\tilde u,B(\theta,r+\delta)) &= MS^p(u, B_{r,r+\delta}(\theta)) + MS^p(w, \overline B(\theta,r))\\
		& \leq MS^p(u, B_{r,r+\delta}(\theta)) + MS^p(w, B(\theta,2r))\\
		&\leq (1+c)MS^p(u, B_{r,r+\delta}(\theta)),
	\end{align*}
where $c>0$ is the same constant as in \eqref{s:MS-2r-2}.	

\noindent
The desired extension operator $T_{\theta,r} \colon  GSBV^p(B_{r,r+\delta}(\theta)) \longrightarrow GSBV^p(B(\theta,r+\delta))$ is then the one associating to any $u\in GSBV^p(B_{r,r+\delta}(\theta))$ the function $\tilde u$ defined by \eqref{ext:1}.

\medskip

\noindent
\textbf{Case 2: $r\geq \delta$.} Since $r < r_\ast$, we have that
	\begin{equation*}
		B(\theta,r(1+ \delta/r_\ast)) \subset B(\theta,r+\delta).
	\end{equation*} 
We divide the proof into two steps. In a first step we extend $u$ from $B_{r,r+\delta}(\theta)$ to $B_{r_\delta,r+\delta}(\theta)$, for some suitably defined $r_\delta<\delta$ (see \eqref{def:Ndelta}-\eqref{def:rdelta}). Then, in a second step, since $r_\delta<\delta$, we can argue as in Case 1 and conclude. 

\smallskip

\textit{Step 1: Dyadic extensions.}  For $i\in \N$ we define the open dyadic annuli 
$$
A_i:= B(\theta,r(1+ \delta/r_\ast)^{2-i})\setminus \overline B(\theta,r(1+ \delta/r_\ast)^{1-i}),
$$
and their semi-open versions
$$
A'_i:=  \overline B(\theta,r(1+ \delta/r_\ast)^{2-i})\setminus  \overline B(\theta,r(1+ \delta/r_\ast)^{1-i}).
$$
Note indeed that for every $i$ the ratio between the outer and inner radii of $A_i$ is constant and equal to $1+\frac{\delta}{r_\ast}$.  
Let $N_\delta \in \N$ be given by
\begin{equation}\label{def:Ndelta}
N_\delta:= \bigg\lfloor\frac{\ln\left(\frac{r_\ast}{\delta}\right)}{\ln(1+\delta/r_\ast)}\bigg\rfloor + 1.
\end{equation}
Taking into account that $r<r_\ast$ it is easy to check that 
\begin{equation}\label{def:rdelta}
r_\delta:=r(1+\delta/r_\ast)^{-N_\delta}<\delta;
\end{equation}
in particular, since in this case $\delta \leq r$, this implies that $B(\theta,r_\delta) \subset B(\theta,r)$. 

We now extend $u$ from $A_1=B(\theta,r(1+ \delta/r_\ast)) \setminus \overline B(\theta,r)$ to $B(\theta,r(1+ \delta/r_\ast)) \setminus \overline B(\theta,r_\delta)$ iteratively. To this end, for $i=1$ we set $u_1:=u_{|A_1}$; then we define the function $v_1:=T_1u_1 \in GSBV^p(A_1\cup A'_2)$, where $T_1$ denotes the extension operator from $GSBV^p(A_1)$ to $GSBV^p(A_1 \cup A'_2)$ provided by Theorem \ref{estfixdomintro}. Hence, $v_1=u_1$ a.e. in $A_1$ and $MS^p(v_1,A_1 \cup A'_2)\leq c\, MS^p(u_1,A_1)$, where the constant $c>0$ depends only on $n$, $p$, $\delta$ and $r_\ast$.  

For $i=2$ we set $u_2:={v_1}_{|A_2}$ and we define $v_2:=T_2 u_2 \in GSBV^p(A_2\cup A'_3)$, where $T_2$ denotes the extension operator from $GSBV^p(A_2)$ to $GSBV^p(A_2 \cup A'_3)$ provided again by Theorem \ref{estfixdomintro}. Therefore  
$v_2=u_2$ a.e. in $A_2$ and $MS^p(v_2,A_2 \cup A'_3)\leq c\, MS^p(u_2,A_2)$, where the constant $c>0$ is the same as in the step $i=1$. Thus we have
\begin{align*}
MS^p(v_2,A_2 \cup A'_3) &\leq c MS^p(u_2,A_2) \leq c MS^p(u_2,A_1 \cup A'_2) \\
&\leq c^2 MS^p(u_1,A_1)=c^2 MS^p(u,A_1).
\end{align*}
Then, by repeating the same procedure as above for every $i=1,\ldots, N_\delta$, we eventually construct $N_\delta$ functions $v_i:=T_i u_i \in GSBV^p(A_i \cup A'_{i+1})$, where $T_i$ denotes the extension operator from $GSBV^p(A_i)$ to $GSBV^p(A_i \cup A'_{i+1})$. Hence, for every $i=1,\ldots, N_\delta$, we have that
$v_i=u_i$ a.e. in $A_i$ and 
$$
MS^p(v_i,A_i\cup A'_{i+1}) \leq c^i MS^p(u,A_1). 
$$
Thus recalling that $A_1:=B(\theta,r(1+ \delta/r_\ast)) \setminus \overline B(\theta,r)$ we get 
\begin{equation}\label{stima-iterata}
MS^p(v_i,A_i\cup A'_{i+1}) \leq c^i MS^p(u,B(\theta,r(1+ \delta/r_\ast)) \setminus \overline B(\theta,r)) \leq c^i MS^p(u,B_{r,r+\delta}(\theta)) 
\end{equation}
for every $i=1,\ldots, N_\delta$.

We now define the function $\hat u \in GSBV^p(B_{r_\delta,r+\delta}(\theta))$ as  
$$
\hat u:= \begin{cases}
u & \text{in}\; B_{r,r+\delta}(\theta)
\cr
v_i & \text{in}\; A'_{i+1}, \; \text{for}\; i=1,\ldots, N_\delta.
\end{cases}
$$
By the definition of $\hat u$ and by \eqref{stima-iterata} we get
\begin{align*}
MS^p(\hat u, B_{r_\delta, r+\delta}(\theta))
&\leq MS^p(u,B(\theta,r+\delta)\setminus \overline B(\theta,r(1+ \delta/r_\ast)))+ \sum_{i=1}^{N_\delta} MS^p(v_i, A_i\cup A'_{i+1})\\
&\leq MS^p(u,B_{r,r+\delta}(\theta))+  \sum_{i=1}^{N_\delta}c^i MS^p(u, B_{r,r+\delta}(\theta))\\
& = \underbrace{\left(1+ \sum_{i=1}^{N_\delta}c^i\right)}_{=: \hat c(n,p,\delta,r_\ast)}MS^p(u, B_{r,r+\delta}(\theta)).
\end{align*}

\smallskip

\textit{Step 2: Extension to $B(\theta,r+\delta)$.} To conclude we need to extend $\hat u$ from $B_{r_\delta,r+\delta}(\theta)$ to $B(\theta,r+\delta)$. Since by construction $r_\delta< \delta<r$, we can follow the same procedure as in Case 1 to extend (the restriction of) $\hat u$ from $B_{r_\delta,r_\delta+\delta}(\theta)$ to $B(\theta,r_\delta+\delta)$. That is, we consider the extended function $\tilde u$ as in \eqref{ext:1} (with $r_\delta$ instead of $r$).  Then $\tilde u \in GSBV^p(B(\theta,r_\delta+\delta))$,  $\tilde u = \hat u$ $\mathcal L^n$-a.e. $B_{r_\delta,r_\delta+\delta}(\theta)$, and
\begin{align*}
		MS^p(\tilde u,B(\theta,r_\delta+\delta)) &\leq (1+c)MS^p(\hat u, B_{r_\delta,r_\delta+\delta}(\theta)) \leq (1+c)MS^p(\hat u, B_{r_\delta,r+\delta}(\theta)).
\end{align*}
The desired extension operator $T_{\theta,r} \colon  GSBV^p(B_{r,r+\delta}(\theta)) \longrightarrow GSBV^p(B(\theta,r+\delta))$ is then the operator associating to any $u\in GSBV^p(B_{r,r+\delta}(\theta))$ the function $\overline u: B(\theta,r+\delta) \to \R$ defined as
$$
\overline{u}:= 
\begin{cases}
\hat u & \text{in}\; B_{r_\delta,r+\delta}(\theta)
\cr
\tilde u & \text{in}\; \overline B(\theta,r_\delta),
\end{cases}
$$
which satisfies 
\begin{align*}
MS^p(\overline u,B(\theta,r+\delta)) &\leq MS^p(\hat u,B_{r_\delta,r+\delta}(\theta)) + MS^p(\tilde u,\overline B(\theta,r_\delta))
\\
&\leq MS^p(\hat u,B_{r_\delta,r+\delta}(\theta)) + MS^p(\tilde u,B(\theta,r_\delta+\delta)) \\
&\leq (2+c) MS^p(\hat u,B_{r_\delta,r+\delta}(\theta)) \leq (2+c)c(n,p,\delta,r_\ast) MS^p(u, B_{r,r+\delta}(\theta)).
\end{align*}

\end{proof}

\noindent We now make use of Lemma \ref{t:ext-cell} to prove the desired $GSBV$-extension result from $A\setminus \e K$ to $A$.  

\begin{theo}[$GSBV$-extension in $A\setminus \e K$]
Let $A \in \mathcal A$, let $K\subset \R^n$ satisfy \eqref{def:P} and \eqref{separation-K}, and let $\e>0$. Let $p>1$. 
Then there exists an extension operator $T_\e : GSBV^p(A \setminus \e K) \longrightarrow GSBV^p(A)$ and a constant $c=c(n,p,\delta,r_\ast)>0$ such that 

\smallskip

$(E1)$\; $T_\e u = u\; \ \mathcal{L}^n\text{-a.e. in } \;  A \setminus \e K,$

\smallskip
$(E2)$\; $MS^p(T_\e u, A) \leq c\big(MS^p(u, A\setminus \e K)+\HH(\partial A)\big)$

\smallskip

\noindent for every $u \in GSBV^p(A \setminus \e K)$. Moreover, the constant $c$ is invariant under homotheties and translations. 

If in addition $u\in L^\infty(A \setminus \e K)$, then 

\smallskip

$(E3)$\; $T_\e u \in L^\infty(A)$, and $\|T_\e u\|_{L^\infty(A)} = \|u\|_{L^\infty(A \setminus \e K)}$.
\label{stochastic-extension-lemma-SBV}
\end{theo}

\begin{proof}
Let $\bar{u}: \R^n\setminus \e K\to \R$ denote the trivial extension of $u$ to $\R^n \setminus \e K$\ie
\begin{equation*}
\bar{u}:=
\begin{cases}
u \quad & \mbox{in } A\setminus \e K\\
0 \quad & \mbox{in }  (\R^n\setminus A)\setminus \e K.
\end{cases}
\end{equation*}
Then $\bar{u} = u$ a.e. in $A\setminus \e K$, and 
\begin{equation}\label{tlMf}
MS^p(\bar{u}, \R^n\setminus \e K) \leq MS^p(u,A\setminus \e K) + \mathcal{H}^{n-1}(\partial A). 
\end{equation}
Let $\mathcal{I}_{\varepsilon}$ be the set of indices $j \in \I_\e$ such that  $\e \overline B(\theta_j,r_j)$ intersects $A$. For $j\in \mathcal{I}_\e$ we use the shorthand $A_j$ for the open annulus $B_{r_j, r_j+\delta}(\theta_{j})$, 
and we denote with $T_{j,\e}: GSBV^p\big(\e A_j\big) \longrightarrow GSBV^p(\e B(\theta_j,r_j+\delta))$ the extension operator provided by Lemma \ref{t:ext-cell}. 
Finally, we define the function  $\tilde u_{\varepsilon}: \R^n\to \R$ as 
\begin{equation*}
\tilde u_{\e}:= 
\begin{cases}
\Big(T_{j,\varepsilon}\big(\bar{u}_{|\e A_j}\big)\Big) &\quad \textrm{in } \e B(\theta_j,r_j+\delta), \, j\in \mathcal{I}_\e,\\
\bar u &\quad \textrm{otherwise}. 
\end{cases}
\end{equation*}
Clearly $\tilde u_{\varepsilon} \in GSBV^p(A\cup \bigcup_{j\in \mathcal{I}_\e} \e B(\theta_j,r_j+\delta))$. Moreover, 
\begin{align*}
MS^p\bigg(\tilde u_\e, A\cup \bigcup_{j\in \mathcal{I}_\e} \e B(\theta_j,r_j+\delta)\bigg) 
&\leq \sum_{j\in \mathcal{I}_\e} MS^p\left(T_{j,\varepsilon}\big(\bar{u}_{|\e A_j}\big), \e B(\theta_j,r_j+\delta)\right) + MS^p(\bar u,\R^n\setminus \e K)\\
&\leq c(n,p,\delta,r_\ast) \sum_{j=1}^{N_\e} MS^p\left(\bar{u}, \e A_j\right) + MS^p(\bar u,\R^n\setminus \e K)\\
& \leq (c(n,p,\delta,r_\ast)+1) \left(MS^p(u,A\setminus \e K) + \mathcal{H}^{n-1}(\partial A)\right),
\end{align*}
where we have used  \eqref{tlMf}, and the fact that, since for each of the operators $T_{j,\e}$ the constant provided by Lemma \ref{t:ext-cell}  
is invariant under translations and homotheties, it is in particular independent of $j$ and $\varepsilon$. 
Finally, the claim follows by defining  $T_\varepsilon u:= \tilde u^{\varepsilon}_{|A}$.
\end{proof}

\begin{rem}\label{stima:interna}
A careful inspection of the proof of Theorem~\ref{stochastic-extension-lemma-SBV} shows that, as in \cite{ACDP}, one can obtain the following estimate, alternative to $(E2)$:
\begin{equation*}
MS^p(T_\e u, A') \leq c MS^p(u, A\setminus \e K), \quad \forall\, A'\in \mathcal{A}, A'\subset \subset A.
\end{equation*}
Indeed, the additional boundary contribution in $(E2)$ is due to the possible presence of perforations that are cut by $\partial A$, and for which the extension result Lemma \ref{t:ext-cell} does not apply. 
This boundary term is clearly no longer necessary if we accept to control the Mumford-Shah of the extended function only far from the boundary.
\end{rem}

\begin{rem}
In Theorem \ref{stochastic-extension-lemma-SBV} it is not necessary to assume that the connected components of $K$ are balls. For instance, the case where each component of 
$K$ is a smooth strictly convex domain does not essentially differ from the case of spherical inclusions. 
\end{rem}

For later use we also state the analogue of Lemma \ref{t:ext-cell} for Sobolev functions (Lemma \ref{Sob:ext-cell}) and for partitions (Lemma \ref{Part:ext-cell}).

\begin{lemma}[Sobolev-extension in an annulus]\label{Sob:ext-cell}
Let $n\geq 2$ and $p>1$; let $\delta >0$ and $r_\ast>\delta$ be fixed. Let $\theta\in \R^n$ and $0<r<r_\ast$; then there exist an extension operator $T_{\theta, r} : W^{1,p}(B_{r,r+\delta}(\theta)) \rightarrow W^{1,p}(B(\theta, r+\delta))$ and a constant $c=c(n,p,\delta, r_\ast)>0$ such that
	\begin{align*}
	T_{\theta,r} u &= u \ \ \mathcal{L}^n\text{-a.e. in } \ B_{r,r+\delta}(\theta),
	\\
	\|T_{\theta,r} u\|_{L^p(B(\theta,r+\delta))}  &\leq c \,\|u\|_{L^p(B_{r,r+\delta}(\theta))},
	\\
	\|D(T_{\theta,r} u)\|_{L^p(B(\theta,r+\delta))}  &\leq c \,\|Du\|_{L^p(B_{r,r+\delta}(\theta))},
	\end{align*}
for every $u \in W^{1,p}(B_{r,r+\delta}(\theta))$. The constant $c$ is invariant under translations and homotheties. 
\end{lemma}

\begin{proof}
The proof can be obtained by repeating every step of the proof of Lemma \ref{t:ext-cell}, up to invoking the extension result \cite[Lemma 2.6]{ACDP} instead of Theorem \ref{estfixdomintro}.
\end{proof}


\begin{lemma}[Extension of a partition in an annulus]\label{Part:ext-cell}
Let $n\geq 2$, and let $\delta >0$ and $r_\ast>\delta$ be fixed. Let $\theta\in \R^n$ and $0<r<r_\ast$; then there exist an extension operator $T_{\theta, r} : \mathcal{P}(B_{r,r+\delta}(\theta)) \rightarrow \mathcal{P}(B(\theta, r+\delta))$ and a constant $c=c(n,\delta, r_\ast)>0$ such that
	\begin{align*}
	T_{\theta,r} u &= u \ \ \mathcal{L}^n\text{-a.e. in } \ B_{r,r+\delta}(\theta),
	\\
	\HH(S_{T_{\theta,r} u}\cap B(\theta,r+\delta))
	&\leq c \HH(S_u\cap B_{r,r+\delta}(\theta)),
	\end{align*}
for every $u \in \mathcal{P}(B_{r,r+\delta}(\theta))$. The constant $c$ is invariant under translations and homotheties. 
\end{lemma}

\begin{proof}
The proof is obtained by combining an adaptation of the proof of Theorem \ref{estfixdomintro} (\cite[Theorem 1.1]{CS}) with the proof of Lemma \ref{t:ext-cell}.
\medskip

\noindent
\textbf{Case 1: $r<\delta$.} In this case we extend from $B_{r, 2r}(\theta)\subset B_{r,r+\delta}(\theta)$ to $B(\theta,2r)$. Up to a translation and a rescaling we reduce to extending a partition $v$ from $B_{1,2}(0)$ to $B(0,2)$. Let $\Phi: B_{\frac12,1}(0)\to B_{1,\frac32}(0)$ denote the reflection map with $\Phi=\textrm{Id}$ on $\partial B(0,1)$, which associates to a point $z\in B_{\frac12,1}(0)$ the point $\tilde z \in B_{1,\frac32}(0)$ on the line joining $z$ with $0$, with $(z+\tilde z)/2\in \partial B(0,1)$. Then the function 
$$
\tilde v:= 
\begin{cases}
v \quad & \text{in } B_{1,2}(0)\\
v\circ \Phi & \text{in } B_{\frac12,1}(0)
\end{cases}
$$
satisfies $\tilde v\in \mathcal{P}(B_{\frac12,2}(0))$ and
\begin{equation}\label{1-vv2}
\HH(S_{\tilde v}\cap B_{\frac12,2}(0)) \leq c \HH(S_{v}\cap B_{1,2}(0)),
\end{equation}
where $c>0$ is a constant depending only on the dimension $n$. Finally, we modify $\tilde v$ in the annulus $B_{\frac12,1}(0)$, and substitute it with a minimiser of the perimeter. More precisely, we let $\hat v\in \mathcal{P}(B_{\frac12,2}(0))$ be a solution of the following minimisation problem 
$$
 \inf\big\{\HH(S_{w}\cap B_{\frac12,2}(0): \ w\in L^1_{\rm{loc}}(\R^n), w_{|B_{\frac12,2}(0)}\in \mathcal P(B_{\frac12,2}(0)), \ w = v \ \text{ in } B_{1,2}(0)\big\}.
$$
Then,  \eqref{1-vv2} gives
\begin{equation}\label{vv2}
\HH(S_{\hat v}\cap B_{\frac12,2}(0)) \leq \HH(S_{\tilde v}\cap B_{\frac12,2}(0)) \leq c \HH(S_{v}\cap B_{1,2}(0)).
\end{equation}
We now distinguish the cases of a ``small'' or ``large'' jump set of $\hat v$ in the annulus $B_{\frac12,1}(0)$. We say that $\hat v$ has a small jump set if 
\begin{equation}\label{vv1}
\HH(S_{\hat v}\cap B_{\frac12,1}(0)) \leq \frac{\gamma}{2^{n-1}},
\end{equation}
where $\gamma=\gamma(n)>0$ is the universal constant as in Lemma \ref{fracture-lemma} (applied with $\tau=1$). We note that the function $\hat v$ satisfies the assumptions of Lemma \ref{fracture-lemma} in $B_{\frac12,1}(0)$. Indeed, $(H1)$ follows by the local minimality of $\hat v$ in the annulus, and $(H2)$ is exactly \eqref{vv1}. Therefore Lemma \ref{fracture-lemma} (with $r=s=r_0=s_0=\frac12$) yields the existence of $\bar{r} \in (\frac46, \frac56)$ such that 
$$
S_{\hat v} \cap \partial B(0,\bar{r}) = \emptyset,
$$ 
namely the trace of ${\hat v}$ is constant on $\partial B(0,\bar{r})$. We denote this constant value by $m$, 
and we define the function $\bar v$ in $B(0,2)$ as 
		\begin{equation*}
			\bar v :=  
			\begin{cases} 
			\displaystyle
			{\hat v}
			& \mbox{in } B_{\bar r, 2}(0), \cr 
			m
			& \mbox{in } \overline B(0,\bar{r}). 
			\end{cases}
		\end{equation*}	
Then $\bar v\in \mathcal{P}(B(0,2))$ and, by \eqref{vv2},
\begin{align*}
\HH(S_{\bar v}\cap B(0,2)) = \HH(S_{\hat v}\cap B_{\bar r, 2}(0)) \leq \HH(S_{\hat v}\cap B_{\frac12,2}(0))  \leq c \HH(S_{v}\cap B_{1,2}(0)).
\end{align*}
Hence the function $\bar v$ is the required extension.

If instead \eqref{vv1} is not satisfied, then the extension is obtained by simply filling the perforation with, \textit{e.g.}, the constant value $0$. In so doing the additional perimeter 
created by the discontinuity on $\partial B(0,1)$ has a comparable perimeter to $\frac{\gamma}{2^{n-1}}$, up to a multiplicative constant. More precisely, we set
		\begin{equation*}
			\bar v:=  
			\begin{cases} 
			\displaystyle
			 v
			& \mbox{in } B_{1,2}(0), \cr 
			0
			& \mbox{in } \overline{B}(0,1). 
			\end{cases}
		\end{equation*}	
	Clearly $\bar v\in \mathcal P(B(0,2))$, and
			\begin{align}
			\mathcal{H}^{n-1}(S_{\bar v} \cap B(0,2))
			&\leq  \mathcal{H}^{n-1}(S_{ v} \cap B_{1,2}(0)) + s_{n}\nonumber \\ 
			&<  \mathcal{H}^{n-1}(S_{v} \cap B_{1,2}(0))+\frac{s_{n} 2^{n-1}}{\gamma}\mathcal{H}^{n-1}(S_{{\hat v}} \cap B_{\frac12,1}(0))
			\nonumber
			\\ \label{stima-J2}
			&\leq 
			c  \, \mathcal{H}^{n-1}(S_{v} \cap B_{1,2}(0)),
		\end{align}
where $s_{n}:=\mathcal H^{n-1}(\partial B(0,1))$ and $c=c(n)>0$. Hence also in this case the function $\bar v$ is the required extension.

\medskip

\noindent
\textbf{Case 2: $r>\delta$.} Since $r < r_\ast$, we have that
	\begin{equation*}
		B(\theta,r(1+ \delta/r_\ast)) \subset B(\theta,r+\delta).
	\end{equation*} 
We now extend from $B(\theta,r(1+ \delta/r_\ast))\setminus \overline B(\theta,r)$ to $B(\theta,r)\setminus \overline B(\theta,r(1+ \delta/r_\ast)^{-1})$. Up to a translation and a rescaling, we can restrict our attention to the case $\theta=0$ and $r=1$\ie to extend from the set $A_1:= B(0,(1+ \delta/r_\ast))\setminus \overline B(0,1)$ to $A_2:= B(0,1)\setminus \overline B(0,(1+ \delta/r_\ast)^{-1})$. Let $v\in \mathcal P(A_1)$; then by denoting with $\Phi: A_2\to A_1$ the reflection map with $\Phi=\textrm{Id}$ on $\partial B(0,1)$, we have that  the function 
$$
\tilde v:= 
\begin{cases}
v \quad & \text{in } A_1\\
v\circ \Phi & \text{in } A_2
\end{cases}
$$
satisfies $\tilde v\in \mathcal{P}(A_1\cup A_2')$, where $A_2' := A_2\cup \partial B(0,1)$, and
\begin{equation}
\HH(S_{\tilde v}\cap (A_1\cup A_2')) \leq c \HH(S_{v}\cap A_1),
\end{equation}
with $c = c(n,\delta, r_\ast)>0$. Again, as in Step 1, we denote with $\hat v \in \mathcal{P}(A_1\cup A'_2)$ a minimiser of the perimeter in $A_1\cup A'_2$ such that 
$\hat v = v$ in $A_1$. We then apply Lemma \ref{fracture-lemma} to obtain the desired extension. Since $A_2 = B(0,1)\setminus B(0,\frac{r^*}{r^*+\delta})$, we have that 
$r=\frac{r^*}{r^*+\delta}$ and $s=\frac{\delta}{r^*+\delta}$ (and note that $s\leq r$ since $\delta<r^*$). 

In this case we say that $\hat v$ has a small jump set in $A_2$ if 
\begin{equation}\label{2:vv1}
\HH(S_{\hat v}\cap A_2) \leq \gamma \left(\frac{\delta}{r^*+\delta}\right)^{n-1},
\end{equation}
where $\gamma=\gamma(n)>0$ is the universal constant as in Lemma \ref{fracture-lemma} (applied with $\tau=1$). We note that the function $\hat v$ satisfies the assumptions of Lemma \ref{fracture-lemma} in $A_2$. Therefore Lemma \ref{fracture-lemma} (with $r_0=r$ and $s_0=s$) yields the existence of $\bar{r} \in \frac13(\frac{3r^*+\delta}{r^*+\delta}, \frac{3r^*+2\delta}{r^*+\delta})$ such that 
$$
S_{\hat v} \cap \partial B(0,\bar{r}) = \emptyset,
$$ 
namely the trace of ${\hat v}$ is constant on $\partial B(0,\bar{r})$, with value, say, $m\in \{0,1\}$. Proceeding as in the previous step yields the conclusion.

\end{proof}


\subsection{Compactness}\label{sect:cpt} In this subsection we use Theorem \ref{stochastic-extension-lemma-SBV} to prove that a sequence 
$(u_\e)$ with equibounded energy $E_\e(\omega)$ can be replaced, without changing the energy, with a sequence which is precompact with respect to the strong 
$L^1$-convergence.


\begin{propo}[Compactness]\label{prop:comp} 
Let $\omega\in \Omega$ and $A\in \mathcal A$ be fixed. Let $(u_\e) \subset L^1(A)$ be a sequence satisfying 
\begin{equation}\label{p:comp-bd}
\sup_{\e>0} \Big(E_\e(\omega)(u_\e,A)+\|u_\e\|_{L^\infty (A\setminus \e K(\omega))}\Big)<+\infty.
\end{equation}
Then there exist a sequence $(\tilde u_\e) \subset SBV^p(A)$ and a function $u\in SBV^p(A)$ such that $\tilde u_\e = u_\e$ $\mathcal L^n$-a.e. in $A\setminus \e K(\omega)$ and (up to a subsequence) $\tilde u_\e \to u$ strongly in $L^1(A)$. 
\end{propo}

\begin{proof}
We start observing that \eqref{p:comp-bd} yields $(u_{\varepsilon}) \subset SBV^p(A \setminus \e K(\om))\cap L^\infty(A\setminus \e K(\om))$. 
Let $T^\omega_{\e}$ be the extension operator from $A\setminus \e K(\om)$ to $A$ as in Theorem \ref{stochastic-extension-lemma-SBV} and set
$$
\tilde{u}_{\varepsilon}:= T^\omega_{\e}\big({u_{\varepsilon}}_{|A \setminus \varepsilon K(\omega)}\big).
$$
Then $\tilde u_{\varepsilon} \in SBV^p(A)\cap L^\infty(A)$, $\tilde{u}_{\varepsilon} = u_{\varepsilon}$ a.e. in $A \setminus \varepsilon K(\omega)$, and (E2) gives
\begin{align}
	MS^p(\tilde u_{\e}, A) & \leq c(n,p,\delta,r_\ast) \left(MS^p(u_{\e}, A \setminus \e K (\omega)) + \HH(\partial A)\right)
	\nonumber \\
	& = \frac{c(n,p,\delta,r_\ast)}{c_1\wedge c_3} \left(E_{\e}(\omega)(u_{\e}, A)+ \HH(\partial A)\right).
	\label{stochastic-uniform-bound-energy}
\end{align} 
Since moreover by (E3) the extension operator $T^\omega_{\e}$ preserves the $L^\infty$-norm, by combining \eqref{p:comp-bd} and \eqref{stochastic-uniform-bound-energy} we immediately deduce that
$$
\sup_{\e>0}\big(MS^p(\tilde u_{\e}, A)+\|\tilde u_{\e}\|_{L^\infty (A)}\big)<+\infty.
$$
Therefore by  Ambrosio's Compactness Theorem \cite[Theorem 4.8]{AFP}, up to subsequences not relabelled, $\tilde{u}_{\varepsilon} \to u$ strongly in $L^{1}(A)$, for some $u \in SBV^p(A)$. 

\end{proof}

\begin{rem}[Weak coerciveness]\label{r:weak-conv}
Let $\om\in \Om$ be fixed and let $(u_\e) \subset L^1(A)$ be such that 
\begin{equation*}
\sup_{\e>0} \Big(E_\e(\omega)(u_\e,A)+\|u_\e\|_{L^\infty (A\setminus \e K(\omega))}\Big)<+\infty.
\end{equation*}
Then, for $P$-a.e. $\om\in \Om$, up to a subsequence not relabelled, we have
\begin{equation}\label{weak-conv}
u_\e \chi_{(\R^n \setminus \e K(\omega))} = u_\e d_\e(\om,\cdot) \wto u \, d(\om) \quad \text{weakly in $L^1(A)$},
\end{equation}
for some $u\in SBV^p(A)$, with $d(\om)$ is as in Definition \ref{dens-sd}.

Indeed, Proposition \ref{prop:comp} yields the existence of a sequence $(\tilde u_\e)\subset SBV^p(A)$ and a function $u \in SBV^p(A)$ such that $\tilde u_\e = u_\e$ in $A\setminus \e K(\om)$ and (up to a subsequence not relabelled) 
\begin{equation}\label{rem:formula}
\tilde{u}_{\varepsilon} \to u \quad  \text{strongly in $L^{1}(A)$}.
\end{equation} 
On the other hand, by the Birkhoff's Ergodic Theorem (see Remark \ref{rem:B}) for $P$-a.e. $\om \in \Om$ we have 
\begin{equation}\label{rem:formula1}
\chi_{(\R^n\setminus \e K(\omega))} = d_\e(\om,\cdot) \wto d(\om) \quad \text{weakly$^*$ in $L^\infty(A)$.}
\end{equation}
Then the conclusion follows from the equality $u_\e \chi_{ (\R^n \setminus \e K(\omega))}=\tilde u_\e \chi_{(\R^n\setminus \e K(\omega))}$, by combining  \eqref{rem:formula} and \eqref{rem:formula1}.  
\end{rem}

		
\section{Homogenisation result}		
\noindent In this section we prove both the existence of the homogenisation formulas defining $f_{\rm hom}$ and $g_{\rm hom}$ and the almost sure $\Gamma$-convergence of $E_\e(\om)$ towards $E_{\rm hom}(\om)$ stated in Theorem \ref{statement-main-thm}.

The existence of the homogenisation formulas is achieved in two steps. The first step consists in applying \cite[Theorem 3.12]{CDMSZ2} to a coercive perturbation of $E_\e$. 
Then in the second step we pass to the limit in the perturbation parameter and show that this procedure leads to $f_{\rm hom}$ and $g_{\rm hom}$. This last step requires the separate extension results for Sobolev functions (Lemma \ref{Sob:ext-cell}) and for partitions (Lemma \ref{Part:ext-cell}). 
		
\begin{theo}[Homogenisation formulas]\label{en-density_vs}
Let $f$ and $g$ be stationary random volume and surface integrands, and let $D\subset \R^n$ be a random perforated domain as in Definition \ref{ass:random-domain}. Assume that the stationarity of $f$, $g$, and $D$ is satisfied with respect to the same group $(\tau_y)_{y\in \R^n}$ of $P$-preserving transformations on $(\Om,\T,P)$. For $\om\in \Om$, let $F(\om)$ and $G(\om)$ be as in \eqref{Effe} and \eqref{Gi}, respectively. Let moreover $m^{1,p}_{F(\om)}$ and $m^{\mathrm{pc}}_{G(\om)}$ be defined by \eqref{emme} and \eqref{emmeG}, respectively.
Then there exists $\Om'\in \mathcal T$, with $P(\Om')=1$,
such that for every $\om\in \Om'$, for every $x, \xi \in \R^{n}$, and every $\nu\in \Sph^{n-1}$ the limits  
\begin{align*}
\lim_{t\to +\infty} \frac{m_{F(\om)}(\ell_\xi,Q_t(tx))}{t^{n}} \quad \textrm{and} \quad \lim_{t\to +\infty} \frac{m^{\mathrm{pc}}_{G(\om)}(u_{t x,1,\nu},Q^\nu_t(tx))}{t^{n-1}}
\end{align*}
exist and are independent of $x$. More precisely, there exist a $(\T\otimes \B^n)$-measurable function $f_{\mathrm{hom}} \colon \Om \times \R^n \to [0,+\infty)$ and a $(\T \otimes \B(\Sph^{n-1}))$-measurable function $g_{\mathrm{hom}} \colon \Om \times \Sph^{n-1} \to [0,+\infty)$ such that, for every $x \in \mathbb{R}^n$,
$\xi \in \mathbb{R}^{n}$, and $\nu \in \mathbb{S}^{n-1}$
\begin{equation}\label{f-hom}
f_{\mathrm{hom}}(\om,\xi)=\lim_{t\to +\infty} \frac{1}{t^n} m^{1,p}_{F(\om)}(\ell_\xi,Q_t(tx))=  \lim_{t\to +\infty} \frac{1}{t^n} m^{1,p}_{F(\om)} (\ell_\xi, Q_t(0)),
\end{equation}
\begin{equation}\label{g-hom}
g_{\mathrm{hom}}(\om,\nu)=\lim_{t\to +\infty} \frac{1}{t^{n-1}}m^{\mathrm{pc}}_{G(\om)}(u_{t x,1,\nu},Q^\nu_t(tx))= \lim_{t\to +\infty} \frac{1}{t^{n-1}}m^{\mathrm{pc}}_{G(\om)}(u_{0,1,\nu},Q^\nu_t(0)).
\end{equation}
If, in addition, $f$, $g$, and $D$ are ergodic, then $f_{\mathrm{hom}}$ and $g_{\mathrm{hom}}$ are independent of $\om$, and
\begin{equation*}
f_{\mathrm{hom}}(\xi)=\lim_{t\to +\infty}\, \frac{1}{t^n} {\int_\Om m^{1,p}_{F(\om)} (\ell_\xi, Q_t(0))\,d P(\om)}, 
\end{equation*}
\begin{equation*}
g_{\mathrm{hom}}(\nu)= \lim_{t\to +\infty} \frac{1}{t^{n-1}} \int_\Om m^{\mathrm{pc}}_{G(\om)} (u_{0,1,\nu}, Q_t^\nu(0)) \, d P(\om).
\end{equation*}
\end{theo}

\begin{proof}
For $k\in \N^*$ we set $f^k(\omega,x,\xi):=a^k(\omega,x)f(\omega,x,\xi)$ and $g^k(\omega,x,\nu):=a^k(\omega,x)g(\omega,x,\nu)$, where 
\begin{equation}
		a^{k}(\omega,x)
		:= \begin{cases}
		1 & \mbox{if} \ x \in \R^n \setminus K(\omega), \\
		\frac{1}{k} & \mbox{if} \ x \in K(\omega),
		\label{ak}
		\end{cases}
	\end{equation}
and consider the coercive functionals $F^k(\om), G^k(\om): \times L^1_{\rm{loc}}(\R^n) \times \mathcal{A} \longrightarrow [0,+\infty]$ defined as
\begin{equation*}
		F^k(\omega)(u,A)
		\displaystyle
		:= \begin{cases}
			\displaystyle
			\int_{A}f^k\left(\omega,x,\nabla u\right)dx & \mbox{if} \ u_{|A} \in W^{1,p}(A), \\
			+\infty & \mbox{otherwise,}
		\end{cases}
	\end{equation*}
	and
	\begin{equation*}
G^k(\omega)(u,A)
		\displaystyle
		:= \begin{cases}
			\displaystyle
			\int_{S_u\cap A}g^k\left(\omega,x,\nu_u\right)\, d\HH & \mbox{if} \ u_{|A} \in GSBV^p(A), \\
			+\infty & \mbox{otherwise.}
		\end{cases}
			\end{equation*}	
Moreover, we denote with $m^{1,p}_{F^k(\om)}$ and $m^{\rm pc}_{G^k(\om)}$ the corresponding minimisation problems as in \eqref{emme} and \eqref{emmeG}, respectively. 

For every fixed $k\in \N^*$ the functions $f^k$ and $g^k$ satisfy the assumptions of \cite[Theorem 3.12]{CDMSZ2}. Hence we can deduce the existence of a set $\Omega^k \subset \Omega$, with $\Omega^k \in \mathcal T$ and  $P(\Omega^k)=1$, such that for every $\omega \in \Omega^k$ and for every $x, \xi\in \R^n, \nu \in \mathbb S^{n-1}$ it holds
\begin{equation}\label{fhom-k}
		\lim_{t \rightarrow +\infty}
		\dfrac{m^{1,p}_{F^{k}(\omega)}(\ell_{\xi}, Q_{t}(tx))}{t^n}=
		\lim_{t \rightarrow +\infty} 
		\dfrac{m^{1,p}_{F^{k}(\omega)}(\ell_{\xi}, Q_{t}(0))}{t^n}=:f^k_{\rm hom}(\om,\xi)
	\end{equation}
and 
	\begin{equation}\label{ghom-k}
		\lim_{t \rightarrow +\infty} 
		\dfrac{m_{G^{k}(\omega)}^{\rm pc}(u_{tx,1,\nu}, Q_{t}^\nu(tx))}{t^{n-1}}=
		\lim_{t \rightarrow +\infty} 
		\dfrac{m_{G^{k}(\omega)}^{\rm pc}(u_{0,1,\nu}, Q_{t}^\nu(0))}{t^{n-1}} =: g_{\rm hom}^{k}(\om,\nu). 
	\end{equation} 	
Furthermore, $f^k_{\rm hom}$ is $(\T \otimes \B^n)$-measurable while $g^k_{\rm hom}$ is $(\T\otimes \B(\Sph^{n-1}))$-measurable. 	
Now we set 
\begin{equation}\label{Omega-prime}
\Omega':=\bigcap_{k\in \N^*} \Omega^k;
\end{equation} 
clearly $\Omega'\in \mathcal T$, $P(\Omega')=1$, and for every $\omega\in \Omega'$ and every $k\in \N^*$, the limits in \eqref{fhom-k} and \eqref{ghom-k} exist.	
We note moreover that for every $\om\in \Om'$, $\xi\in \R^n$, and $\nu\in \Sph^{n-1}$ the sequences $f^k_{\rm hom}(\om,\xi)$ and $g_{\rm hom}^{k}(\om,\nu)$ are decreasing in $k$. Therefore, for every $\om\in \Om'$, $\xi\in \R^n$, and $\nu\in \Sph^{n-1}$ we define the functions $f_{\rm hom}$ and $g_{\rm hom}$ as follows:
\begin{equation}\label{def-fhom}
\lim_{k \to +\infty}f^k_{\rm hom}(\om,\xi)=\inf_{k\in \N^*} f^k_{\rm hom}(\om,\xi)=:f_{\rm hom}(\om,\xi) 
\end{equation}
and 
\begin{equation}\label{def-ghom}
\lim_{k \to +\infty}g^k_{\rm hom}(\om,\nu)=\inf_{k\in \N^*} g^k_{\rm hom}(\om,\nu)=:g_{\rm hom}(\om,\nu).
\end{equation}
By definition, we clearly have that $f_{\rm hom}$ is $(\T \otimes \B^n)$-measurable and $g_{\rm hom}$ is $(\T\otimes \B(\Sph^{n-1}))$-measurable. 
We now show that the functions $f_{\rm hom}$ and $g_{\rm hom}$ satisfy \eqref{f-hom} and \eqref{g-hom}, respectively.

\smallskip

For every $\om\in \Om'$, $x,\xi\in \R^n$, and $\nu\in \Sph^{n-1}$ set
\begin{align*}
&\overline{f}(\om, x, \xi) := \limsup_{t \rightarrow +\infty}\dfrac{m^{1,p}_{F(\omega)}(\ell_\xi,Q_{t}(tx))}{t^{n}},\\
&\underline f(\om, x, \xi) := \liminf_{t \rightarrow +\infty}\dfrac{m^{1,p}_{F(\omega)}(\ell_\xi,Q_{t}(tx))}{t^{n}},
\end{align*}
and 
\begin{align*}
&\overline{g}(\om, x,\nu) := \limsup_{t \rightarrow +\infty}\dfrac{m^{\rm pc}_{G(\omega)}(u_{tx,1,\nu},Q_{t}(tx))}{t^{n-1}},\\
&\underline{g}(\om, x,\nu) := \liminf_{t \rightarrow +\infty}\dfrac{m^{\rm pc}_{G(\omega)}(u_{tx,1,\nu},Q_{t}(tx))}{t^{n-1}}.
\end{align*}
Then, to conclude it is enough to show that 
\begin{equation}\label{claim-f}
\overline{f} =\underline{f} = f_{\rm hom}
\end{equation}
and 
\begin{equation}\label{claim-g}
\overline{g}=\underline{g} = g_{\rm hom},
\end{equation}
with $f_{\rm hom}$ and $g_{\rm hom}$ as in \eqref{def-fhom} and \eqref{def-ghom}, respectively. 
We prove the two claims above in two separate steps. 

\medskip

\textit{Step 1: Proof of \eqref{claim-f}.} By definition $0\leq f\chi_{A\setminus K(\om)}\leq f^k$ for every $k\in \N^*$, hence by the monotonicity of the integral we immediately deduce that $\overline{f}\leq f^k_{\rm hom}$ for every $k\in \N^*$. Therefore
\begin{equation}\label{stima-1-f}
\overline{f}(\om,x, \xi)\leq \inf_{k \in \N^*} f^k_{\rm hom}(\om,\xi)= f_{\rm hom}(\om,\xi),
\end{equation}
for every $\om\in \Om'$, $x,\xi\in \R^n$. 

We now show that $f_{\rm hom} \leq \underline f$. To this end let $t\gg1$, $\om\in \Om'$, $x\in \R^n$ and $\xi\in \R^n$ be fixed. For $\eta>0$ let $\hat u\in W^{1,p}(Q_t(tx))$ be such that $\hat u=\ell_\xi$ near $\partial Q_{t}(tx)$ and 
$$
F(\om)(\hat u,Q_t(tx))\leq m^{1,p}_{F(\om)}(\ell_\xi,Q_t(tx)) + \eta t^n.
$$ 
Since $\ell_\xi$ is a competitor for $m^{1,p}_{F(\om)}(\ell_\xi,Q_t(tx))$ we immediately get
\begin{align}\label{stima-u-hat}
\int_{Q_t(tx)\setminus K(\omega)}|D\hat u|^{p}\dy&\leq \frac{1}{c_1} F(\om)(\hat u, Q_t(tx))\nonumber \\
&\leq\frac{1}{c_1} \Big((m^{1,p}_{F(\om)}(\ell_\xi,Q_t(tx))+\eta t^n \Big) \leq \frac{1}{c_1}\Big(c_2(1+|\xi|^p) t^n+\eta t^n\Big).
\end{align}
Starting from $\hat u$ we now construct a competitor for $m^{1,p}_{F^k(\om)}(\ell_\xi,Q_t(tx))$. First of all, we extend $\hat u$ by setting $\hat u = \ell_\xi$ in $\R^n\setminus Q_t(tx)$. Now, let
$\mathcal J \subset \I$ denote the set of indices $j$ such that $ B(\theta_{j}(\omega),r_{j}(\omega)) \cap Q_t(tx)\neq \emptyset$. We clearly have 
\begin{equation*}
Q_t(tx)\subset Q_t(tx)\cup \bigcup_{j\in \mathcal J} B(\theta_j(\omega), r_j(\omega)+\delta).
\end{equation*}
For every $j\in \mathcal J$ we set $\hat u_j:=\hat{u}_{|A_j(\omega)}$, where $A_j(\om)$ denotes the open annulus $B_{r_j(\om), r_j(\om)+\delta}(\theta_j(\om))$. By applying Lemma \ref{Sob:ext-cell} in every $A_j(\omega)$ we deduce the existence of an extension operator $T_{j}^\omega : W^{1,p}(A_j(\omega))\longrightarrow W^{1,p}(B(\theta_j(\omega),r_j(\omega)+\delta)$ and a constant $c>0$ independent of $j$ and $\omega$, such that
	\begin{equation*}
		\|D\big(T_{j}^\omega\hat{u}_j\big)\|_{L^{p}(B(\theta_{j}(\omega),r_{j}(\omega)+\delta)}
		\leq c\|D\hat{u}_j\|_{L^p(A_j(\omega))}.
	\end{equation*}  
We then define the function $\tilde{u} : \R^n \rightarrow \mathbb{R}$ as follows
	\begin{equation*}
		\tilde{u} 
		= \sum_{j \in \mathcal J}(T_{j}^\omega\hat{u}_j)\chi_{B(\theta_{j}(\omega),r_{j}(\omega)+\delta)} + \hat u\, \chi_{Q_t(tx)\setminus K^\delta(\omega)},
	\end{equation*}
where 
$$
K^\delta(\omega):= \bigcup_{j\in \I}  B(\theta_j(\omega), r_j(\omega)+\delta).
$$
By construction $\tilde u_{|Q_t(tx)} \in W^{1,p}(Q_t(tx))$. Moreover,
$$
\|D \tilde u\|_{L^p(Q_t(tx))} \leq c \|D \hat u\|_{L^p(Q_t(tx)\setminus K(\om))},
$$
therefore from \eqref{stima-u-hat} we deduce that
\begin{equation}\label{stima-u-hat-2}
\int_{Q_t(tx)}|D\tilde u|^{p}\dy\leq  c\int_{Q_t(tx) \setminus K(\omega)}|D\hat u|^{p}\dy \leq \frac{c}{c_1}\Big(c_2(1+ |\xi|^p) t^n+\eta t^n\Big).
\end{equation}
We note that in general the function $\tilde u$ does not coincide with $\ell_\xi$ in a neighbourhood of $\partial Q_{t}(tx)$, since we might have altered the boundary value in the perforations intersecting $\partial Q_t(tx)$. We then need to further modify $\tilde u$ in a way such that it attains the boundary datum. To this aim, let $\varphi \in C^\infty_0(Q_{t}(tx))$ be a cut-off function between $Q_{t-4(r_\ast+\delta)}(tx)$ and $Q_{t}(tx)$\ie $0\leq \varphi\leq 1$, $\varphi\equiv 1$ in $Q_{t-4(r_\ast+\delta)}(tx)$, $\varphi\equiv 0$ in $\R^n \setminus Q_{t}(tx)$, and $\|D\varphi\|_{\infty}\leq c$, with $c=c(n, r_\ast, \delta)>0$. 
Set 
$$
w:= \varphi \,\tilde u+(1-\varphi) \ell_\xi;
$$
clearly $w\in W^{1,p}(Q_{t}(tx))$, and $w=\ell_\xi$ in a neighbourhood of $\partial Q_{t}(tx)$. 
We now claim that 
\begin{equation}\label{eq:Dw0}
\lim_{t\to +\infty} \, \frac{1}{t^n} \int_{Q_t(tx)\setminus Q_{t-4(r_\ast+\delta)}(tx)}|Dw|^p dy = 0.
\end{equation}
To ease the notation, in what follows we set $t':=t-4(r_\ast+\delta)$. 
We clearly have
\begin{align}\label{eq:Dw}
\int_{Q_t(tx)\setminus Q_{t'}(tx)}|Dw|^p dy &\leq c\int_{Q_t(tx)\setminus Q_{t'}(tx)}|D\tilde u|^p dy \nonumber\\
&+ c \int_{Q_t(tx)\setminus Q_{t'}(tx)}|\tilde u- \ell_\xi|^p dy + c |\xi|^p t^{n-1}.
\end{align}
We now cover $Q_t(tx)\setminus Q_{t'}(tx)$ with a finite number of possibly overlapping cubes with side-length $2(r_\ast+\delta)$, having one face on the boundary $\partial Q_t(tx)$. Thus we
write
$$
Q_t(tx)\setminus Q_{t'}(tx) = \bigcup_{\sigma\in \mathcal{S}} Q^\sigma_{2(r_\ast+\delta)},
$$
where $Q^\sigma_{2(r_\ast+\delta)}:=\sigma+Q_{2(r_\ast+\delta)}$ and $\mathcal{S}\subset \R^n$ is a finite set of translation vectors 
such that the volume of this covering is asymptotically equal to the volume of $Q_t(tx)\setminus Q_{t'}(tx)$, for $t\to +\infty$.

We now apply the Poincar\'e inequality to the function $\tilde u- \ell_\xi$ in $Q_t(tx)\setminus Q_{t'}(tx)$. 
To do so we preliminarily observe that for every $\sigma\in \mathcal S$ it holds
\begin{equation}\label{capacity}
\HH\left(\partial Q_{2(r_\ast+\delta)}^\sigma\cap \{\tilde u = \ell_\xi\}\right)\geq \delta^{n-1}.
\end{equation}
This is clearly true if $\partial Q_{2(r_\ast+\delta)}^\sigma\cap K(\omega)=\emptyset$, since in that case $\tilde u=\ell_\xi$ on the whole face $\partial Q_{2(r_\ast+\delta)}^\sigma\cap \partial Q_t(tx)$, whose $\HH$-measure is larger than $\delta^{n-1}$. If instead $\partial Q_{2(r_\ast+\delta)}^\sigma\cap K(\omega)\neq\emptyset$, since each ball in $K(\omega)$ has diameter smaller than $2r_\ast$ and is separated from any other ball by a distance which is at least $2\delta$, inequality \eqref{capacity} holds in this case as well. Therefore 
the Poincar\'e inequality applied in every cube $Q^\sigma_{2(r_\ast+\delta)}$ gives 
\begin{equation*}
\int_{Q^\sigma_{2(r_\ast+\delta)}}|\tilde u- \ell_\xi|^p dy \leq  C\int_{Q^\sigma_{2(r_\ast+\delta)}}|D\tilde u- \xi|^p dy,
\end{equation*}
where $C=C(n,p,\delta,r_\ast)>0$ is independent of $\sigma$. Hence by adding up all the cubes $Q^\sigma_{2(r_\ast+\delta)}$, with $\sigma \in \mathcal{S}$, we get
\begin{equation}\label{poincare}
\int_{Q_t(tx)\setminus Q_{t'}(tx)}|\tilde u- \ell_\xi|^p dy \leq C \int_{Q_t(tx)\setminus Q_{t'}(tx)}|D\tilde u- \xi|^p dy.
\end{equation}
Then, gathering \eqref{eq:Dw} and \eqref{poincare} yields 
$$
\int_{Q_t(tx)\setminus Q_{t'}(tx)}|Dw|^p dy \leq c\int_{Q_t(tx)\setminus Q_{t'}(tx)}|D\tilde u|^p dy +c |\xi|^p t^{n-1}.
$$
Hence to prove \eqref{eq:Dw0} it is enough to show that 
$$
\lim_{t\to +\infty} \frac{1}{t^n}\int_{Q_t(tx)\setminus Q_{t'}(tx)}|D\tilde u|^p dy =0.
$$
The latter is a consequence of the equality 
$$
 \frac{1}{t^n}\int_{Q_t(tx)\setminus Q_{t'}(tx)}|D\tilde u|^p dy = \int_{Q_1(x)\setminus Q_{1-\frac{4(r_\ast+\delta)}t}(x)}|D v|^p dz,
$$
where $v(z):= \frac1t \tilde u(tz)$ for every $z\in Q_1(x)$. In fact, by \eqref{stima-u-hat-2} we have  
\begin{equation*}
\int_{Q_1(x)}|Dv|^{p}d z = \frac{1}{t^n}\int_{Q_t(tx)}|D\tilde u|^{p}\dy\leq  \frac{c}{c_1} \Big(c_2(1+ |\xi|^p)+\eta\Big),
\end{equation*}
thus
$$
\lim_{t\to +\infty} \int_{Q_1(x)\setminus Q_{1-\frac{4(r_\ast+\delta)}t}(x)}|D v|^p dz=0,
$$
by the absolute continuity of the Lebesgue integral.

Since $w$ is a competitor for $m^{1,p}_{F^k(\om)}(\ell_\xi, Q_{t}(tx))$, by invoking \eqref{stima-u-hat-2} we find
\begin{align*}
	 \frac{m^{1,p}_{F^k(\om)}(\ell_\xi, Q_{t}(tx))}{t^n} 
	 &\leq \dfrac{1}{t^{n}} F^k(\omega)(w,Q_{t}(tx))\\
	& \leq \dfrac{1}{t^{n}}F(\om)(\hat u, Q_{t'}(tx)) + \frac{c_2}{k t^n} \int_{Q_{t}(tx)}(1+ |D\tilde u|^{p}) \dy \\
	 &\quad +\dfrac{c_2}{t^{n}}\int_{Q_{t}(tx)\setminus Q_{t'}(tx)}\!\!\!\!(1+|Dw|^{p}) \dy
	\\
	& \leq \dfrac{1}{t^{n}}F(\om)(\hat u, Q_{t}(tx))+  \frac{c}{k}(1+|\xi|^p)
	 +\dfrac{c_2}{t^{n}}\int_{Q_{t}(tx)\setminus Q_{t'}(tx)}\!\!\!\!(1+|Dw|^{p}) \dy\\
	& \leq  \dfrac{m^{1,p}_{F}(\omega)(\ell_\xi,Q_{t}(tx))}{t^{n}} +\eta +  \frac{c}{k}(1+|\xi|^p)
	 +\dfrac{c_2}{t^{n}}\int_{Q_{t}(tx)\setminus Q_{t'}(tx)}\!\!\!\!(1+|Dw|^{p}) \dy.
\end{align*} 
Therefore, by and \eqref{eq:Dw0}, passing to the liminf as $t \to +\infty$ we get 
$$
f^k_{\rm hom}(\om,\xi) \leq 
\underline{f}(\om,x, \xi)+ \eta+ \frac{c}{k}(1+|\xi|^p),
$$
for every $\om \in \Om'$, $x,\xi\in \R^n$, and $k\in \N^*$.
Thus letting $k \to +\infty$ yields
\begin{equation}\label{stima-2-f}
f_{\rm hom}(\om, \xi)= \inf_{k \in \N^*} f^k_{\rm hom}(\om,\xi) \leq 
\underline{f}(\om, x, \xi) +\eta
\end{equation}
for every $\om\in \Om'$ and $x,\xi\in \R^n$. Hence, by the arbitrariness of $\eta>0$, gathering \eqref{stima-1-f} and \eqref{stima-2-f} eventually gives \eqref{claim-f} and thus \eqref{f-hom}.

\medskip

\textit{Step 2: Proof of \eqref{claim-g}.} By definition $0 \leq g\chi_{A\setminus K(\om)}\leq g^k$ for every $k\in \N^*$, hence by the monotonicity of the integral we immediately deduce that $\overline{g}\leq g^k_{\rm hom}$ for every $k\in \N^*$. Therefore
\begin{equation}\label{stima-1-g}
\overline{g}(\om, x, \nu)\leq \inf_{k \in \N^*} g^k_{\rm hom}(\om, \nu)= g_{\rm hom}(\om, \nu),
\end{equation}
for every $\om\in \Om'$, $x\in \R^n$ and $\nu\in \Sph^{n-1}$. 

We now show that $g_{\rm hom} \leq \underline g$. To this end let $t\gg1$, $\om\in \Om'$, $x\in \R^n$ and $\nu\in \Sph^{n-1}$ be fixed. For $\eta>0$ let $\hat u\in \mathcal P(Q^\nu_t(tx))$ be such that $\hat u= u_{tx,1,\nu}$ near $\partial Q^\nu_{t}(tx)$ and 
\begin{equation}\label{min:hatu}
G(\om)(\hat u,Q^\nu_t(tx)) \leq m^{\rm pc}_{G(\om)}(u_{tx,1,\nu},Q^\nu_t(tx)) + \eta t^{n-1}.
\end{equation}
Since $u_{tx,1,\nu}$ is a competitor for $m^{\rm pc}_{G(\om)}(u_{tx,1,\nu},Q^\nu_t(tx))$, we immediately get
\begin{align}\label{stima-u-hat-surf}
\frac{\mathcal H^{n-1}(S_{\hat u}\cap (Q_{t}^{\nu}(tx) \setminus K(\omega)))}{t^{n-1}}&\leq \dfrac{G(\om)(\hat u,Q^\nu_t(tx))}{c_3 t^{n-1}}\nonumber \\
&\leq 
\frac{m^{\rm pc}_{G(\om)}(u_{tx,1,\nu},Q^\nu_t(tx))}{c_3t^{n-1}} + \frac{\eta}{c_3} \leq \frac{c_4+\eta}{c_3}.
\end{align}
We now modify $\hat u$ in order to obtain a competitor for $m^{\rm pc}_{G^k(\om)}(u_{tx,1,\nu},Q^\nu_t(tx))$. We preliminarily extend $\hat u$ to the whole $\R^n$ by setting $\hat u = u_{tx,1,\nu}$ 
in $\R^n \setminus Q^\nu_t(tx))$. Now, we denote with $\mathcal J \subset \I$ the set of indices $j$ such that $B(\theta_{j}(\omega),r_{j}(\omega))  \cap Q^\nu_t(tx)\neq \O$. 
For each $j\in \mathcal J$ we set $\hat u_{j}:=\hat{u}_{|A_j(\om)}$, with $A_j(\om):=B_{r_j(\om), r_j(\om)+\delta}(\theta_j(\om))$. 

We divide the proof into three substeps.

\smallskip

\textit{Substep 2.1: Extension of $\hat u$ in the inner perforations.} Let $\mathcal J_I\subset\mathcal J$ denote the set of indices $j$ such that $ B(\theta_{j}(\omega),r_{j}(\omega)+\delta) \subset Q_{t}^{\nu}(tx)$. By Lemma \ref{Part:ext-cell} there exists an extension $v_j:= T_j \hat u_j\in \mathcal{P}(B(\theta_{j}(\omega),r_{j}(\omega)+\delta))$ of $\hat u_j$ whose jump set in $B(\theta_{j}(\omega),r_{j}(\omega)+\delta)$ is controlled, in measure, by the jump set of $\hat u_j$ (and hence by the jump of $\hat u$ in $A_j(\om)$).

\smallskip

\textit{Substep 2.2: Modification of $\hat u$ in the boundary perforations.} Let $ \mathcal J_B:= \mathcal J\setminus \mathcal J_I$, and let $j\in \mathcal J_B$. In order to preserve the boundary conditions we need to distinguish between two cases. We say that $j\in \mathcal J^+_B$ if $B(\theta_{j}(\omega),r_{j}(\omega))\cap \partial( Q_t^\nu(tx) \cap \{(y-tx)\cdot \nu>0\}) \neq \emptyset$, and set $\mathcal J_B^-:=\mathcal J_B\setminus \mathcal J_B^+$.

If $j\in \mathcal J_B^+$ we set 
\begin{equation*}
w_{+,j}:= \begin{cases} 
			\displaystyle
			\hat{u}_j
			& \mbox{in } A_j(\om), \cr 
			1
			& \mbox{in } \overline B(\theta_{j}(\omega),r_{j}(\omega)),
	       \end{cases}
\end{equation*}	
	while if $j\in \mathcal J_B^-$ we set 
\begin{equation*}
w_{-,j}:= \begin{cases} 
			\displaystyle
			\hat{u}_j
			& \mbox{in } A_j(\om), \cr 
			0
			& \mbox{in } \overline B(\theta_{j}(\omega),r_{j}(\omega)).
	     \end{cases}
\end{equation*}			
Clearly, for every $j\in \mathcal J_B$, the additional jump created by $w_{\pm,j}$ is controlled by the perimeter of the boundary perforations $B(\theta_{j}(\omega),r_{j}(\omega))$. Since the perforations in $K(\omega)$ are pairwise disjoint (and in particular this is true for the boundary perforations), the total additional jump due to the boundary perforations is controlled by the perimeter of $Q_t^\nu(tx)$\ie it is equal to $c\, t^{n-1}$ for some $c>0$ independent of $t$.

\medskip

\textit{Substep 2.3: Adding up all the contributions.} We now denote with $\tilde u\in\mathcal{P}(Q_t^\nu(tx))$ the function defined as 
\begin{equation*}
\tilde u:=
\begin{cases}
\hat u &\quad \textrm{in }\, Q_t^\nu(tx)\setminus K^\delta(\om),\\ \smallskip
v_{j} &\quad \textrm{in }\, B(\theta_{j}(\omega),r_{j}(\omega)+\delta), j\in \mathcal{J}_{I}\\ \smallskip
w_{\pm,j} &\quad \textrm{in }\, B(\theta_{j}(\omega),r_{j}(\omega)+\delta)\cap Q_t^\nu(tx), j\in \mathcal{J}_{B}.
\end{cases}
\end{equation*}
By construction the function $\tilde u$ satisfies the following properties:	

\smallskip

	\begin{enumerate}
		\item[a.] $\tilde u = u_{tx,1,\nu}$ in a neighbourhood of $\partial Q^\nu_t(tx)$;
		
		\smallskip
		
		\item[b.] $\mathcal{H}^{n-1}(S_{\tilde u} \cap (Q_t^{\nu}(tx) \setminus K(\omega))) \leq  \mathcal{H}^{n-1}(S_{\hat u} \cap (Q_t^{\nu}(tx) \setminus K(\omega)))$;
		
		\smallskip
		
		\item[c.] $\mathcal{H}^{n-1}(S_{\tilde u} \cap Q_t^{\nu}(tx)) \leq c (\mathcal{H}^{n-1}(S_{\hat u} \cap (Q_t^{\nu}(tx) \setminus K(\omega))) + t^{n-1})$, for some $c>0$ independent of $t$.
			\end{enumerate}

\smallskip

\noindent
Since $\tilde u$ is a competitor for $m^{\rm pc}_{{G^k}(\om)}(u_{tx,1,\nu},Q^\nu_t(tx))$, by combining b., c., and \eqref{stima-u-hat-surf} we get
	\begin{align*}
		\dfrac{m_{G^{k}(\omega)}^{\textrm{pc}}(u_{tx,1,\nu},Q_{t}^{\nu}(tx))}{t^{n-1}} 
		&\leq \dfrac{1}{t^{n-1}}{G^k}(\om)(\tilde u, Q_t^\nu(tx))\\ 
		&\leq \dfrac{1}{t^{n-1}}G (\om)(\tilde u, Q_t^\nu(tx))
		+ \dfrac{c_4}{k\,t^{n-1}}\mathcal{H}^{n-1}({Q_{t}^{\nu}(tx) \cap S_{\tilde{u}}}) \\
		&\leq \dfrac{1}{t^{n-1}} {G} (\om)(\hat u, Q_t^\nu(tx)) + \frac{c}{k}\\
		& \leq  \dfrac{m^{\rm pc}_{{G}(\omega)}(u_{tx,1,\nu},Q_{t}^{\nu}(tx))}{t^{n-1}} + \eta+ \frac{c}{k},
 \end{align*}
where we have also used \eqref{min:hatu}. Therefore passing to the liminf as $t \to +\infty$ we get 
$$
g^k_{\rm hom}(\om, \nu) \leq 
\underline{g}(\om, x, \nu)+ \eta+ \frac{c}{k},
$$
for every $\om\in \Om'$, $x\in\R^n$, $\nu \in \Sph^{n-1}$, and $k\in \N^*$.
Thus finally letting $k \to +\infty$ and then $\eta\to 0$ yields
\begin{equation}\label{stima-2-g}
g_{\rm hom}(\om, \nu):= \inf_{k \in \N^*} g^k_{\rm hom}(\om, \nu) \leq 
\underline{g}(\om, x, \nu),
\end{equation}
for every $\om \in \Om'$, $x \in \R^n$ and $\nu \in \Sph^{n-1}$. Hence gathering \eqref{stima-1-g} and \eqref{stima-2-g} eventually gives \eqref{claim-g} and thus \eqref{g-hom}.

\medskip

If $f$, $g$, and $K$ are stationary with respect to an ergodic group of $P$-preserving transformations, then \cite[Theorem 3.12]{CDMSZ2} 
ensures that $f_{\rm hom}^k$ and $g^k_{\rm hom}$ (and hence $f_{\rm hom}$ and $g_{\rm hom}$) are independent of $\om$. 
Then, the thesis follows by integrating \eqref{claim-f} and \eqref{claim-g} over $\Om$, by the Dominated Convergence Theorem. 
\end{proof}


\begin{rem}[$\Gamma$-convergence of the perturbed functionals]\label{G-conv-Eeps-k} Let $f$, $g$ and $D$ be as in Theorem \ref{en-density_vs}. 
For $k\in \N^*$ we set $f^k(\omega,x,\xi):=a^k(\omega,x)f(\omega,x,\xi)$ and $g^k(\omega,x,\nu):=a^k(\omega,x)g(\omega,x,\nu)$, where $a^k$ is 
defined as in \eqref{ak}. For $\e>0$ and $k\in \N^\ast$, let $E_\e^k(\om) \colon L^1_{\rm loc}(\R^n)\times \mathcal{A} \longrightarrow (0,+\infty]$ be the functionals defined as
\begin{equation*}
E^k_\e(\omega)(u,A)
		\displaystyle
		:= \begin{cases}
			\displaystyle
			\int_{A}f^k\left(\omega,\frac{x}{\e}, \nabla u\right) dx + \int_{S_u\cap A}g^k\left(\omega,\frac{x}{\e},\nu_u\right)\, d\HH & \mbox{if} \ u_{|A} \in GSBV^p(A), \\
			+\infty & \mbox{otherwise.}
		\end{cases}
\end{equation*}
If $\Om'$ is the set in the statement of Theorem \ref{en-density_vs} (defined as in \eqref{Omega-prime}), we deduce from \cite[Theorem 3.13]{CDMSZ2} that for every $\om\in \Om'$ and $k\in \N^\ast$ the functionals $E^k_\e(\om)$ $\Gamma$-converge to the homogeneous free-discontinuity functional 
$E_{\rm hom}^k(\om) \colon L^1_{\rm loc}(\R^n)\times \mathcal{A} \longrightarrow (0,+\infty]$ given by
\begin{equation}\label{E_hom-k}
E^k_{\rm hom}(u,A)
		\displaystyle
		:= \begin{cases}
			\displaystyle
			\int_{A}f^k_{\rm hom}(\om,\nabla u)\dx + \int_{S_u\cap A}g^k_{\rm hom}(\om,\nu_u)\, d\mathcal H^{n-1} & \mbox{if} \ u_{|A} \in GSBV^p(A), \\
			+\infty & \mbox{otherwise,}
		\end{cases}
\end{equation}
where $f^k_{\rm hom}$ and $g^k_{\rm hom}$ are as in \eqref{fhom-k} and \eqref{ghom-k}, respectively.
\end{rem}

\begin{theo}[$\Gamma$-convergence]\label{theo:stochMS} Let $f$ and $g$ be stationary random volume and surface integrands, and let $D\subset \R^n$ be a random perforated domain as in Definition \ref{ass:random-domain}. Assume that the stationarity of $f$, $g$ and $D$ is satisfied with respect to the same group $(\tau_y)_{y\in \R^n}$ of $P$-preserving transformations on $(\Om,\T,P)$. Let $E_\e$ be as in \eqref{def:Eneps}, let $\Om'\in \mathcal T$ (with $P(\Om')=1$), $f_{\rm hom}$, and $g_{\rm hom}$ be as in Theorem \ref{en-density_vs}. Then, for every $\om\in \Om'$ and every $A\in \mathcal{A}$, the functionals $E_\e(\omega)(\cdot, A)$ $\Gamma$-converge in $L^1_{\rm{loc}}(\R^n)$ to the homogeneous functional $E_{\rm hom}(\om) \colon  L^{1}_{\rm loc}(\mathbb{R}^{n})\times \mathcal{A} \longrightarrow [0,+\infty]$ defined as
	\begin{equation}\label{Gamma-limit}
		E_{\rm hom}(u,A)
		\displaystyle
		:= \begin{cases}
		\displaystyle
		\int_{A}f_{\rm hom}(\om,\nabla u) \ dx 
		+\int_{A\cap S_{u}}g_{\rm hom}(\om,\nu_u)\ d\mathcal{H}^{n-1} 
		& \mbox{if} \ u_{|_A} \in GSBV^{p}(A), \\
		+\infty & \mbox{otherwise}.
		\end{cases}
	\end{equation}
Moreover, for every $\om \in \Om'$, $\xi, \xi_1, \xi_2\in \R^n$ and $\nu\in \Sph^{n-1}$ we have that
\begin{equation}\label{f-bds}
\tilde c_0 |\xi|^p\leq f_{\mathrm{hom}}(\om,\xi) \leq c_2(1+ |\xi|^p),
\end{equation}
and 
\begin{equation}\label{g-bds}
\tilde c_0 \leq g_{\mathrm{hom}}(\om,\nu) \leq  c_4,
\end{equation}
where $\tilde c_0=\tilde c_0(n,\delta)>0$, and $c_2$ and $c_4$ are as in $(f2)$ and $(g2)$. Further, there exists $L'>0$ such that
\begin{equation}\label{f-LIP}
|f_{\rm hom}(\om,\xi_1)-f_{\rm hom}(\om,\xi_2)| \leq L'\big(1+|\xi_1|^{p-1}+|\xi_2|^{p-1}\big)|\xi_1-\xi_2|.
\end{equation}

\end{theo}

\begin{proof}
In view of \eqref{def-fhom} and \eqref{def-ghom}, the Monotone Convergence Theorem yields
	\begin{equation}
		E_{\rm hom}(\om)(u,A) 
		= \inf_{k \in \N^*}E_{\rm hom}^{k}(\om)(u,A) 
		= \lim_{k\to +\infty}E_{\rm hom}^{k}(\om)(u,A)
		\label{approximation-ms}
	\end{equation}
for every $\om \in \Omega'$, $A\in\mathcal{A}$ and $u\in GSBV^{p}(A)$, where $E^k_{\rm hom}$ is as in \eqref{E_hom-k}.

We prove the $\Gamma$-convergence of $E_\e$ to $E_{\rm hom}$ in two steps.  

\medskip

\textit{Step 1: liminf-inequality.} Let $\om\in \Om'$ and $A\in \mathcal A$ be fixed. Let $u \in GSBV^p(A)$ and let $(u_\e)\subset L^1_{\rm loc}(\R^n)$ be a sequence satisfying $u_\e \to u$ strongly in $L^1(A)$ and $\sup_\e E_{\varepsilon}(\omega)(u_{\varepsilon},A)<+\infty$. Note in particular $(u_\e)\subset GSBV^p(A)$. For $M>0$ we consider the truncated function $u^M:=(u\wedge M)\vee(-M) \in GSBV^p(A) \cap L^\infty(A)$ and the truncated sequence $(u_\e^M) \subset GSBV^p(A) \cap L^\infty(A)$; clearly $(u_\e^M)$ converges to $u^M$ in $L^1(A)$ as $\e\to 0$.

Let $(\tilde u_\e) \subset SBV^p(A)\cap L^\infty(A)$ be the extension provided by Proposition \ref{prop:comp}, such that $\tilde u_\e = u_\e^M$ a.e. in $A\setminus \e K(\omega)$, and let $\tilde u\in SBV^p(A)\cap L^\infty(A)$ be such that (up to a subsequence) $\tilde u_\e \to \tilde u$ strongly in $L^1(A)$. Since the sequences  $(u_\e)$ and $(\tilde u_\e)$ coincide in $A\setminus \e K(\omega)$, we  deduce by Property \ref{property:d} 
that $\tilde u= u^M$ a.e. in $A$. Furthermore, \eqref{stochastic-uniform-bound-energy} gives 
$$
MS^p(\tilde u_{\e}, A)  \leq \frac{c(n,p,\delta,r_\ast)}{c_1\wedge c_3} \left(E_{\e}(\omega)(u^M_{\e}, A)+ \HH(\partial A)\right),
$$
and therefore we have
\begin{align*}
E^k_\e(\om)(\tilde u_\e, A) &\leq E_{\e}(\omega)(u^M_{\e}, A) +\frac{{c_2\vee c_4}}{k} MS^p (\tilde u_\e, A \cap \e K(\om)) + \frac{c_2}{k} \mathcal{L}^n(A \cap \e K(\om))
\\
& \leq \Big(1+\frac{c}{k}\Big) E_{\e}(\omega)(u^M_{\e}, A)+ \frac{c}{k} \HH(\partial A) +\frac{c_2}{k} \mathcal{L}^n(A),
\end{align*}
where $c=c(n,p,\delta,r_\ast)$. Then, by Remark \ref{G-conv-Eeps-k} we deduce that for every $\om\in \Om'$, $A \in \mathcal A$ and $k\in \N^*$  
\begin{align*}
E^k_{\rm hom}(\om)(u^M, A)&\leq \liminf_{\e \to 0} E^k_\e(\om)(\tilde u_\e, A) \\
&\leq \Big(1+\frac{c}{k}\Big) \liminf_{\e \to 0} E_{\e}(\omega)(u^M_{\e}, A)+ \frac{c}{k} \HH(\partial A) +\frac{c_2}{k} \mathcal{L}^n(A).
\end{align*}
By letting $k \to +\infty$ and using \eqref{approximation-ms}, we then get
\begin{equation}\label{linf:M}
E_{\rm hom}(\om)(u^M, A)= \lim_{k \to +\infty} E^k_{\rm hom}(\om)(u^M, A) \leq \liminf_{\e \to 0} E_{\e}(\omega)(u^M_{\e}, A),
\end{equation}
and hence the liminf-inequality is proved for the truncations, for every $M>0$. Now we observe that $E_\e$ decreases by truncations up to a quantifiable error, namely
$$
E_{\e}(\omega)(u^M_{\e}, A) \leq E_{\e}(\omega)(u_{\e}, A) + \int_{A\cap \{|u_\e|>M\}}\!\!\! f(\om,x,0) dx 
\leq E_{\e}(\omega)(u_{\e}, A) + c_2 \mathcal{L}^n(A\cap \{|u_\e|>M\}).
$$ 
Therefore, from \eqref{linf:M} we obtain the improved estimate 
\begin{align*}
E_{\rm hom}(\om)(u^M, A) &\leq \liminf_{\e \to 0}  \left(E_{\e}(\omega)(u_{\e}, A) + c_2 \mathcal{L}^n(A\cap \{|u_\e|>M\})\right)\\
& \leq \liminf_{\e \to 0}  E_{\e}(\omega)(u_{\e}, A) + c_2   \limsup_{\e \to 0} \mathcal{L}^n(A\cap \{|u_\e|>M\}).
\end{align*}
Since $u_\e \to u$ in $L^1(A)$ we have that  $\limsup_{\e\to 0}\mathcal{L}^n(A\cap \{|u_\e|>M\}) \leq \mathcal{L}^n(A\cap \{|u|>M\})$, and hence 
$$
E_{\rm hom}(\om)(u^M, A) \leq  \liminf_{\e \to 0}  E_{\e}(\omega)(u_{\e}, A) + c_2 \mathcal{L}^n(A\cap \{|u|>M\}).
$$
Finally, since $u^M \to u$ in $L^1(A)$ as $M \to +\infty$, the liminf-inequality follows by the lower semicontinuity of $E_{\rm hom}(\om)(\cdot, A)$. 

\medskip

\textit{Step 2: limsup-inequality.}
Let $\omega \in \Omega'$ and $A \in \mathcal{A}$ be fixed. Let $u \in GSBV^{p}(A)$; in view of Remark \ref{G-conv-Eeps-k} there exists  $(u_{\varepsilon}) \subset GSBV^{p}(A)$ such that $u_\e \to u$ in $L^1(A)$ and $\lim_{\e \to 0}E^k_\e(\om)(u_\e,A)=E^k_{\rm hom}(u,A)$. Then by the definition of $E_\e^k$ we have 
	\begin{equation*}
		E_{\rm hom}^{k}(u,A) 
		= \lim_{\varepsilon \to 0}E_{\varepsilon}^{k}(\omega)(u_{\varepsilon},A)
		\geq \limsup_{\varepsilon\rightarrow 0} E_{\varepsilon}(\omega)(u_{\varepsilon},A),
	\end{equation*}
for every $k\in \N^*$. 	
Then, letting $k\to +\infty$, from \eqref{approximation-ms} we finally deduce 
	\begin{equation*}
		E_{\rm hom}(u,A) 
		= \lim_{k \to +\infty} E_{\rm hom}^{k}(u,A) 
		\geq \limsup_{\varepsilon\to 0} E_{\varepsilon}(\omega)(u_{\varepsilon},A)
	\end{equation*}
and hence the limsup-inequality is proved.

\medskip

\textit{Step 3: Lower bounds on the limit integrands.} We start by proving the lower bound in \eqref{f-bds}. To do so, let $\om \in \Om'$ and $\xi\in \R^n$, and let $u_\e$ be a recovery sequence for $E_\e(\om)$ at $\ell_\xi$ in $Q$. With no loss of generality we can assume that the sequence is bounded in $L^\infty$. Moreover, let $T^\om_\e  u_\e$ denote the extension of $u_\e$ in $Q$ provided by Theorem \ref{stochastic-extension-lemma-SBV}; note that by Ambrosio's Compactness Theorem $T^\om_\e  u_\e$ converges in $L^1$, and since $T^\om_\e  u_\e = u_\e$ in $A\setminus \e K(\om)$, by Property \ref{property:d} we have that $T^\om_\e  u_\e \to \ell_\xi$ in $L^1$, up to a subsequence. By \cite[Theorem 4.7]{AFP}, for every $Q'\subset \subset Q$, we have 
\begin{align*}
MS^p(\ell_\xi, Q') &\leq \liminf_{\e\to 0} MS^p(T^\om_\e  u_\e, Q') \leq c \liminf_{\e\to 0} MS^p(u_\e, Q\setminus \e K(\om)) \\
& \leq \frac{c}{c_1\wedge c_3} \, \liminf_{\e\to 0} E_\e(\om)(u_\e, Q) = \frac{c}{c_1\wedge c_3} E_{\textrm{hom}}(\om)(\ell_\xi, Q),
\end{align*}
where we have also used Remark \ref{stima:interna}. In conclusion,
$$
\mathcal{L}^n(Q')|\xi|^p \leq \frac{c}{c_1\wedge c_3} f_{\textrm{hom}}(\om,\xi),
$$
which gives the lower bound in \eqref{f-bds} for $\tilde c_0:= \frac{c_1\wedge c_3}c$, by letting $Q'\nearrow Q$.

\smallskip 

For the proof of the lower bound in \eqref{g-bds} we proceed similarly. Let $\om \in \Om'$ and $\nu\in \mathbb{S}^{n-1}$, let $u_\e$ be a recovery sequence for $E_\e(\om)$ at $u_{0,1,\nu}$ in $Q^\nu$, and let $T^\om_\e  u_\e$ denote the extension of $u_\e$ in $Q^\nu$ provided by Theorem \ref{stochastic-extension-lemma-SBV}. Again by \cite[Theorem 4.7]{AFP} and by Remark \ref{stima:interna}, for every $Q'\subset \subset Q^\nu$, we have 
\begin{align*}
MS^p(u_{0,1,\nu}, Q') &\leq \liminf_{\e\to 0} MS^p(T^\om_\e  u_\e, Q') \leq c \liminf_{\e\to 0} MS^p(u_\e, Q^\nu\setminus \e K(\om)) \\
& \leq \frac{c}{c_1\wedge c_3} \, \liminf_{\e\to 0} E_\e(\om)(u_\e, Q^\nu) = \frac{c}{c_1\wedge c_3} E_{\textrm{hom}}(\om)(u_{0,1,\nu}, Q^\nu).
\end{align*}
In conclusion,
$$
\HH(S_{u_{0,1,\nu}}\cap Q') \leq \frac{c}{c_1\wedge c_3} g_{\textrm{hom}}(\om,\nu),
$$
which gives the lower bound in \eqref{g-bds} for $\tilde c_0$ defined above, by letting $Q'\nearrow Q^\nu$.

\medskip

\textit{Step 4: Upper bounds on the limit integrands.} The upper bound in \eqref{f-bds} follows immediately by taking, for $\om \in \Om'$ and $\xi\in \R^n$, 
the sequence $u_\e = \ell_\xi$ and by using the liminf inequality for $E_\e$ in $Q$ and the bound $(f2)$, since 
$$
c_2(1+|\xi|^p)\geq \liminf_{\e\to 0} E_\e(\om)(u_\e, Q) \geq E_{\textrm{hom}}(\om)(\ell_\xi, Q) = f_{\textrm{hom}}(\om,\xi). 
$$
The proof of the upper bound in \eqref{g-bds} is completely analogous.
\medskip

\textit{Step 5: Lipschitz continuity of $f_{\textrm{hom}}$.} Property \eqref{f-LIP} follows from the bounds in \eqref{f-bds} and from the convexity of $f_{\textrm{hom}}(\om,\cdot)$.
\end{proof}

\begin{rem}\label{weak:conv-rem}
In Theorem~\ref{theo:stochMS} the $L^1$-topology can be replaced by the weak convergence in \eqref{weak-conv}.
\end{rem}

In view of Remark~\ref{weak:conv-rem}, as a corollary of Theorem~\ref{theo:stochMS} we obtain a $\Gamma$-convergence result for the following (asymptotically degenerate coercive) functionals. 

Let $\e>0$, and let $(\alpha_\e)$ and $(\beta_\e)$ be two positive sequences, infinitesimal as $\e\to 0$. For $\om \in \Om$, $x, \xi \in \R^n$ and $\nu \in \mathbb{S}^{n-1}$ we define 
\begin{equation*}
		a_\e(\omega,x)
		:= \begin{cases}
		1 & \mbox{if} \ x \in \R^n \setminus K(\omega), \\
		\alpha_\e & \mbox{if} \ x \in K(\omega),
		\end{cases}
		\qquad
		b_\e(\omega,x)
		:= \begin{cases}
		1 & \mbox{if} \ x \in \R^n \setminus K(\omega), \\
		\beta_\e & \mbox{if} \ x \in K(\omega),
		\end{cases}
	\end{equation*}
$f_\e(\omega,x,\xi):=a_\e(\omega,x)f(\omega,x,\xi)$ and $g_\e(\omega,x,\nu):=b_\e(\omega,x)g(\omega,x,\nu)$. 
\smallskip

We now consider the functionals $E^{\alpha_\e,\beta_\e}_\e(\om) \colon L^1_{\rm loc}(\R^n)\times \mathcal{A} \longrightarrow (0,+\infty]$  defined as
\begin{equation}\label{E_ep-ep}
E^{\alpha_\e,\beta_\e}_\e(\omega)(u,A)
		\displaystyle
		:= \begin{cases}
			\displaystyle
			\int_{A}f_\e\left(\omega,\frac{x}{\e}, \nabla u\right) dx + \int_{S_u\cap A}g_\e\left(\omega,\frac{x}{\e},\nu_u\right)\, d\HH & \mbox{if} \ u_{|A} \in GSBV^p(A), \\
			+\infty & \mbox{otherwise.}
		\end{cases}
\end{equation}

For an overview on the behaviour of the functionals in \eqref{E_ep-ep} in the deterministic case see \cite{BrSo,Zeppierina}. 

\begin{coro}
Let $\Om'\in \mathcal T$ (with $P(\Om')=1$), $f_{\rm hom}$, and $g_{\rm hom}$ be as in Theorem \ref{en-density_vs}. Then, for every $\om\in \Om'$ and every $A\in \mathcal{A}$, the functionals $E^{\alpha_\e,\beta_\e}_\e(\omega)(\cdot,A)$ in \eqref{E_ep-ep} $\Gamma$-converge with respect to the weak convergence in \eqref{weak-conv} to the homogeneous functional $E_{\rm hom}(\om)(\cdot,A)$ defined in \eqref{Gamma-limit}.
\end{coro}

\begin{proof}
The liminf inequality follows immediately from Theorem \ref{Gamma-limit} and Remark~\ref{weak:conv-rem}, due to the lower bound $E^{\alpha_\e,\beta_\e}_\e \geq E_\e$. Let now $A\in \mathcal A$. For the limsup inequality, by a standard truncation argument we can reduce to the case of $u\in SBV^p(A)\cap L^\infty(A)$. Let $(u_\e)\subset L^1_{\rm loc}(\R^n)$ be a sequence such that $u_\e \to u$ in $L^1_{\rm loc}(\R^n)$ and $\lim_{\e\to 0} E_\e(\om)(u_\e,A) = E_{\rm hom}(\om)(u,A)$. With no loss of generality we can assume that $\|u_\e\|_{L^\infty(A)}\leq \|u\|_{L^\infty(A)}$. Let $(\tilde u_\e) \subset SBV^p(A)\cap L^\infty(A)$ be the extension provided by Theorem \ref{stochastic-extension-lemma-SBV}. By Property \ref{property:d}, $\tilde u_\e \to u$ strongly in $L^1(A)$. Furthermore, 
\begin{multline*}
E^{\alpha_\e,\beta_\e}_{\e}(\omega)(\tilde u_{\e}, A) = 
E_{\e}(\omega)(u_{\e}, A) 
+ 
\alpha_\e\int_{\e K(\om)\cap A}f\left(\omega,\frac{x}{\e}, \nabla \tilde u_\e\right) dx \\+ 
\beta_\e\int_{S_{\tilde u_\e}\cap (\e K(\om)\cap A)}g\left(\omega,\frac{x}{\e},\nu_{\tilde u_\e}\right)\, d\HH.
\end{multline*}
Since 
$$
\int_{\e K(\om)\cap A}f\left(\omega,\frac{x}{\e}, \nabla \tilde u_\e\right) dx \leq c_2\int_{\e K(\om)\cap A}(1+|\nabla \tilde u_\e|^p) dx
\leq c_2 \mathcal L^n(A) + c_2 MS^p(\tilde u_\e,A),
$$
and 
$$
\int_{S_{\tilde u_\e}\cap (\e K(\om)\cap A)}g\left(\omega,\frac{x}{\e},\nu_{\tilde u_\e}\right)\, d\HH \leq c_4 MS^p(\tilde u_\e,A),
$$
by Theorem \ref{stochastic-extension-lemma-SBV} we deduce that 
$$
\lim_{\e\to 0} E^{\alpha_\e,\beta_\e}_{\e}(\omega)(\tilde u_{\e}, A) = \lim_{\e\to 0}E_{\e}(\omega)(u_{\e}, A) = E_{\rm hom}(\om)(u,A).
$$

\end{proof}

\section*{Acknowledgments} 
\noindent This work of C. I. Zeppieri was supported by the Deutsche Forschungsgemeinschaft (DFG, German Research Foundation) project number ZE 1186/1-1 and under the Germany Excellence Strategy EXC 2044-390685587, Mathematics M\"unster: Dynamics--Geometry--Structure.

\end{document}